\documentclass[a4paper]{amsart}

\usepackage{amsthm}
\usepackage{amsmath}
\usepackage{amsfonts}
\usepackage{amssymb}
\usepackage{amsopn}
\usepackage{upgreek}
\usepackage{amscd}
\usepackage{graphicx}
\usepackage{enumerate}
\usepackage[all]{xy}

\theoremstyle{plain}
\newtheorem{thm}{Theorem}[section]

\newtheorem{prop}[thm]{Proposition}
\newtheorem{wniosek}[thm]{Corollary}
\newtheorem{lemma}[thm]{Lemma}

\theoremstyle{definition}
\newtheorem{dfn}[thm]{Definition}

\theoremstyle{remark}
\newtheorem{remark}[thm]{Remark}
\newtheorem{question}[thm]{Question}

\renewcommand{\geq}{\geqslant}

\DeclareMathOperator{\fhi}{\varphi}
\DeclareMathOperator{\Hom}{\mathrm{Hom}}
\DeclareMathOperator{\Aut}{\mathrm{Aut}}

\DeclareMathOperator{\Set}{\mathrm{Set}}
\DeclareMathOperator{\Mod}{\mathrm{Mod}}

\DeclareMathOperator{\Spec}{\mathrm{Spec}}

\DeclareMathOperator{\coker}{\mathrm{coker}}

\DeclareMathOperator{\id}{\mathrm{id}}
\DeclareMathOperator{\maks}{\mathfrak{m}}
\DeclareMathOperator{\HO}{\mathrm{H}}

\DeclareMathOperator{\res}{\mathrm{res}}
\DeclareMathOperator{\ord}{\mathrm{ord}}
\DeclareMathOperator{\ind}{\mathrm{ind}}
\DeclareMathOperator{\infl}{\mathrm{inf}}
\DeclareMathOperator{\im}{\mathrm{im}}
\DeclareMathOperator{\transgresja}{\mathrm{tg}}
\DeclareMathOperator{\Der}{\mathrm{Der}}
\DeclareMathOperator{\kate}{\mathcal{C}}
\DeclareMathOperator{\katez}{\hat{\mathcal{C}}}

\DeclareMathOperator{\der}{\Uptheta}
\DeclareMathOperator{\derq}{\Uptheta^{\sharp}}
\DeclareMathOperator{\derj}{\Uptheta_1}
\DeclareMathOperator{\psl}{\![\![\!}
\DeclareMathOperator{\psr}{\!]\!]}

\newlength{\myVSpace}% the height of the box
\DeclareMathOperator{\Z}{\mathbf{Z}}

\author{Jakub Byszewski}
\title{Deformation functors of local actions}

\address{Instytut Matematyki Uniwersytetu Jagiello\'nskiego\\ ul.\ \L ojasiewicza 6\\ 30-348 Krak\'ow\\ Poland} \email{Jakub.Byszewski@im.uj.edu.pl}
\begin{document}
\begin{abstract}
We study the behaviour of infinitesimal deformation functors of local group actions with regard to passing to subgroups and quotient groups. Inspired by the cohomological information, we conjecture the existence of a decomposition of a deformation functor of a local $G$-action as a smooth extension of a fibered product of functors related to a subgroup and a quotient group of $G$.
\end{abstract}

\maketitle

Let $H$ be a finite group acting on a smooth curve $X$ over a perfect field $k$ of positive characteristic $p>0$. For every closed point $x \in X$, we can consider the action of the stabilizer $G=H_x$ of the point $x$ on the completed local ring of the curve $R=\widehat{\mathcal{O}}_{X,x}$. Since the curve is smooth,  $R$ is a power series ring in one variable over $k$, and we obtain in this way an action $\rho$ of $G$ on $k\psl t \psr$. We call such actions local group actions (of the group $G$). Consider the infinitesimal (local) deformation functor $D_G$ which classifies lifts of $\rho$ to local artinian rings with residue field $k$. In the article, we study the behaviour of the deformation functors $D_G$ with regard to passing to subgroups and quotient groups of $G$. 

Let $N$ be a normal subgroup of $G$. In sections \ref{alians} and \ref{misja} we define a restriction map $$\res \colon D_G \to D_N^{G/N}$$ and an induction map $$\ind \colon D_G \to D_{G/N}.$$ The tangent map to $\res$ coincides with the familiar restriction map in group cohomology $\res \colon \HO^1(G,\der)\to \HO^1(N,\der)^{G/N}.$ A similar cohomological description for the tangent map to $\ind$ is given in Theorem $\ref{przebaczenie}$. The maps $\res$ and $\ind$ provide a ``lift'' of the known cohomological operations to the level of local deformation functors and can be used to study the relation between the functors $D_G$, $D_N$ and $D_{G/N}$. An easy case is the following one: if the order of the group $G/N$ is prime to $p$ and the functor $D_N$ is pro-representable, then the functor $D_{G/N}$ vanishes and the map $D_G \to D_N^{G/N}$ is an isomorphism (Theorem \ref{pyszczek}). 

Similar questions can also be investigated for the global deformation functor $D_{X,G}$ of a curve $X$ with an action of a group $G$. In the global case, $D_{X,G}$ is pro-representable and the map $D_{X,G} \to D_{X,N}^{G/N}$ is an isomorphism under rather weak assumptions (e.g., $\mathrm{genus}(X) \geq 2$ suffices, cf. \cite[p. 1909]{Maugeais}).

The general situation for local deformation functors is much more complicated, and both functors $D_N$ and $D_{G/N}$ play a r\^ole. We pose the following question: Does there exist a pro-representable functor $F$ such that the map $$(\res, \ind) \colon D_G \to D_N^{G/N} \times_F D_{G/N}$$ is smooth? As first evidence, we see that this is true at the cohomological level (cf. Corollary \ref{Jugendstill}). We approach the question as follows: we consider the morphism $$(\res, \ind) \colon D_G \to D_N^{G/N} \times D_{G/N}$$ and study the deformation theory \emph{of this morphism}. In our main result (Theorem \ref{szczezl}), we describe -- under additional assumptions of pro-representability -- a complete obstruction space for this morphism. The construction of this obstruction space is cohomological. In the light of Proposition \ref{wiarygodnosc}, the result provides evidence for a positive answer to the question.

The interrelationship between establishing pro-representability and computing versal deformation rings was already seen in \cite{BC} (cf.\ particularly Remark 3.3 of loc.\ cit.) and was rather a surprise. For a longer discussion of pro-representability of local deformation functors, see the end of section 1.

A positive answer to the posed question would enable an explicit computation of the deformation ring of $D_G$ in terms of the deformation rings of $D_N$ and $D_{G/N}$ (cf. Remark \ref{pazdziernik}). It is conceivable that these ``d\'evissage'' techniques can contribute to an inductive solution to the problem of lifting group actions on curves to characteristic zero as for example in the problem of Oort (cf.\ e.g.\ \cite{CGH}), but at present we do not know of such an application. One of the difficulties is certainly that the knowledge of the complete universal deformation ring is an overkill with respect to the mere question of the existence of a lift to characteristic zero (for which it is enough to know only the characteristic of the versal deformation ring). 

\begin{subsection}*{Acknowledgements}
The work has been done at the University of Utrecht with the financial support of the Netherlands Organisation of Scientific Research under the VIDI project no.\ 52307314 of Gunther Cornelissen and at the Universit\'e de Versailles/Universit\'e Paris-Sud 11 with the financial support of the European Community under the Arithmetic Algebraic Geometry programme.

I'm most grateful to Gunther Cornelissen and Ariane M\'ezard for the discussions, their encouragement and interest in this work. It was Ariane M\'ezard who first asked me the question about the relation between the functors $D_G$, $D_N$ and $D_{G/N}$. I would also like to thank an anonymous referee for helpful comments.
\end{subsection}

\begin{section}{Introduction}

Let $k$ be a perfect field with $\mathrm{char}(k) = p > 0$ and let $G$ be a finite group.

\begin{dfn} A \emph{local} $G$-\emph{action}\index{local $G$-action} is an injective homomorphism $\rho \colon G \to \Aut_k k\psl t \psr$. \end{dfn}

We will study infinitesimal deformation functors of local $G$-actions. These functors are defined on a suitable category of artinian rings. Let $W(k)$ be the ring of ($p$-)Witt vectors over the field $k$ and let $\mathcal{C}$ be the category of local artinian $W(k)$-algebras with residue field $k$ and local morphisms of $W(k)$-algebras. \begin{dfn} A \emph{lift} of $\rho$ to an object $A$ of $\kate$ is a homomorphism $\rho_A \colon G \to \mathrm{Aut}_A A\psl t\psr$ which reduces to $\rho$ modulo $\mathfrak{m}_A$.
Two lifts $\rho_A$, $\rho'_A$ are equivalent if they are conjugate by an element $\chi \in \ker \left( \Aut_A A\psl t \psr \rightarrow \Aut_k k\psl t \psr\right)$, i.e., if for every $g\in G$ we have $\rho'_A(g) = \chi \rho_A(g) \chi^{-1}$. The infinitesimal deformation functor of $\rho$ is a functor $$D_{\rho} \colon \mathcal{C} \to \mathbf{Sets}$$ which maps $A$ to the set of equivalence classes of lifts of $\rho$ to $A$. By abuse of notation, we usually write $D_{G}$ for $D_{\rho}$.\end{dfn}

We use a notational convention typical for deformation theory: whenever we denote an element of a local ring $A$ by $\varepsilon$, we tacitly assume that  $\mathfrak{m}_A \varepsilon=0$. Furthermore, $k[\varepsilon]$ denotes the ring $k[w]/w^2$. A surjection $p \colon A' \to A$ in $\kate$ is called \emph{small} if its kernel $I$ is annihilated by $\mathfrak{m}_{A'}$. In this case $I$ is a finite dimensional vector space over $k$.

Denote by $\der=\Der_k k\psl t \psr$ the module of derivations of the ring $k\psl t \psr$ with the natural action of $G$ given by $d^g=\rho(g)^{-1}d\rho(g), g \in G.$ To simplify the notation, we write $\Gamma_A =\Aut_A A\psl t \psr$ for $A$ in $\kate$. Given a map $A'\to A$, we denote by $\Gamma_{A',A}$ the kernel of the map $\Gamma_{A'} \to \Gamma_A$.

\begin{lemma}\label{potwor} Let $A' \to A$ be a small surjection in $\kate$ with kernel $I$. \begin{enumerate} \item[\textup{(i)}] The map  $m \colon \der \otimes I \to \Gamma_{A',A}$ given by $m(d\otimes\varepsilon)(x)=x+\varepsilon d(\bar{x})$ is an isomorphism. \item[\textup{(ii)}] Let $\gamma \in \Gamma_{A',A}$ and let $\gamma_g\in \Gamma_{A'}$ be such that $\overline{\gamma}_g = \rho(g)$. Then $\gamma_g \gamma \gamma_g^{-1} = \gamma^g,$ where we regard $\Gamma_{A',A}$ as a $G$-module via the map $m$ above. In particular, the subgroups $\Gamma_{A',A}$ and $\Gamma_{A',k}$ commute element-wise.\end{enumerate} \end{lemma}

The proof is immediate.

We call the set $T_D=D(k[\varepsilon])$ the \emph{tangent space} to $D$. A standard calculation shows that it is bijective with the group cohomology $\HO^1(G,\der)$ via a map which associates to a class of a one-cocycle $\gamma_g$ in $\der\otimes k\varepsilon \simeq \Gamma_{k[\varepsilon],k}$ the class of a lift $\rho_A(g)=\gamma_g \rho(g)$. In particular, this set has a natural structure of a finite dimensional vector space over $k$.

Given a small surjection $A' \to A$ in $\kate$ with kernel $I$, we can consider the question whether a given element $\kappa \in D_G(A)$ lies in the image of the map $D(A') \to D(A)$. Choose a morphism $\rho \colon G \to \Aut_A A\psl t \psr$ in the class of $\kappa$. A standard construction produces a two-cocycle $\eta(g,h)=\rho^{*}(g)\rho^*(h)\rho^*(gh)^{-1}$ with values in $\Gamma_{A',A} \simeq \der \otimes I$, where $\rho^* \colon G \to \Aut_{A'} A'\psl t \psr$ is a \emph{set-theoretic} function which is a lift of $\rho$ to $A'$. This cocycle gives an obstruction to lifting $\kappa $ to $D_G(A')$. This shows that the space $\HO^2(G,\der)$ is an obstruction space to the functor $D_G$ (for a precise definition of an obstruction space in a slightly more general context, cf. Definition \ref{sentyment}).

We now briefly recall the concepts of pro-representability and versal hull (for more information, cf.\ \cite{Schlessinger}). Consider the category $\katez$ consisting of local noetherian $W(k)$-algebras with residue field $k$. The category $\kate$ is a full subcategory of $\katez$. A functor $D \colon \kate \to \Set$ is called pro-representable if there exists an object $R$ in $\katez$ such that $D$ is isomorphic to the functor $h_R=\Hom_{\katez}(R,\cdot)$. Such a ring $R$ is then unique and is called the universal deformation ring of $D$. A morphism $D \to E$ of functors $D,E \colon \kate \to \Set$ such that $D(k)$ and $E(k)$ are one-point sets is called smooth if for any surjective morphism $A'\to A$ in $\kate$ the induced map $D(A') \to D(A)\times_{E(A)} E(A')$ is surjective. A functor $D \colon \kate \to \Set$ is said to have a versal hull if there is a smooth morphism $\fhi \colon h_R \to D$ from a pro-representable functor $h_R$ which induces an isomorphism on tangent spaces. If a functor has a versal hull, it is unique (up to a \emph{nonunique} isomorphism). The ring $R$ is called a versal deformation ring. The morphism $\fhi$ is always surjective and it is injective if and only if $D$ is pro-representable. If $A'\to A$ is a small surjection in $\kate$ with kernel $I$ and $D$ has a versal hull, the group $T_D \otimes I$ acts transitively on the fibers of the map $D(A') \to D(A)$. If $D$ is furthermore pro-representable than this action is free making the fibers of $D(A') \to D(A)$ into torsors under the action of $T_D \otimes I$.

It is well-known that the functors $D_G$ satisfy the conditions of Schlessinger \cite{Schlessinger}, and hence have a versal hull. For some of the results, we will need a stronger condition of pro-representability. It might seem at first that this stronger condition is unlikely, since the functors $D_G$ have very many ``infinitesimal automorphisms", whereas pro-representability in algebraic geometry is usually related to ``smallness'' of the space of infinitesimal automorphisms. Nevertheless, the condition has been established to hold in several cases. Recall that a local $G$-action $\rho$ induces on the group $G$ a decreasing filtration of higher ramification groups $G \supseteq G_1 \supseteq G_2 \supseteq \ldots$ with $$G_i=\{\sigma\in G \mid \mathrm{ord}_t(\rho(\sigma)(t)-t)>i\}, \quad i\ge 1.$$ If $G_1$ is zero, we call the action tamely ramified. This happens only if the order of $G$ is coprime with $p$. If $G_2$ is zero, we call the action \emph{weakly ramified}. It has been shown (cf. \cite{Nakajima}) that every local action coming from an action of a group on an ordinary curve is weakly ramified. Pro-representability of weakly ramified local deformation functors has been established except when $p=2$ and $G=\mathbf{Z}/2$ or $G=\mathbf{Z}/2\oplus \mathbf{Z}/2$ (cf.\ \cite{BC}), when it fails. This list of two counterexamples further reduces to only a single one (namely, $G=\Z/2$) if one restricts oneself to the associated equicharacteristic functor. (By this we mean that we restrict the deformation functor to the full subcategory of $\kate$ consisting of these artinian rings which are $k$-algebras. The equicharacteristic deformation functor can be pro-representable without the original functor being pro-representable. This is the case of the weakly ramified $G=\mathbf{Z}/2\oplus \mathbf{Z}/2$ action.) More generally, if $n$ is the smallest integer such that $G_{n+1}=0$, one sometimes says that the Hasse conductor of the action is $n$. A single case of higher conductor has also been resolved. This is when when $p=5$, $G=\mathbf{Z}/5$ and the conductor is two, with the resulting functor being pro-representable as well (cf.\ \cite{BCK}). All this seems to suggest that the assumption of pro-representability is a weak one, though admittedly the evidence is not yet conclusive.
\end{section}

\begin{section}{Restriction}\label{alians}

In this section we study the operation of restricting a local action to a subgroup. 

 Let $\rho$ be a local $G$-action and let $N$ be a normal subgroup of $G$. 
Restricting the action of the group to $N$ we obtain a restriction morphism $\res \colon D_{G} \to D_{N}$.

\begin{dfn} We define an action of $G/N$ on the functor $D_N$. The action of $G$ on the set $D_N(A)$ \begin{align*} G \times D_N(A) &\to D_N(A)\\(g,\kappa) &\mapsto g_{\ast}\kappa\end{align*} is given as follows: Choose a representative $\rho_A\colon N \to \Gamma_A$ of $\kappa \in D_N(A)$ and define $g_{\ast}\kappa$ as the class of $$g_{\ast}\rho_A(n)=r_g \rho_A(g^{-1}ng)  r_g^{-1},$$
where we write $r_g$ for a lift of $\rho(g)$ to $A$, i.e., $\overline{r_g}=\rho(g)$. \end{dfn}

\begin{prop}\label{action} The construction above induces an action of the group $G/N$ on the functor $D_N$. Deformations lying in the image of the restriction map $\res \colon D_G \to D_N$ are invariant under the action of $G/N$.
\end{prop}
\begin{proof} All the properties are straightforward to verify. As an example, we show that $g_{\ast}h_{\ast}=(gh)_{\ast}$ for $g,h\in G$. 
Indeed, we have \begin{align*} g_*(h_*\rho_{A})(n)&=r_g r_h\rho_{A}(h^{-1}g^{-1}ngh)r_h^{-1}r_g^{-1}\\ &=\chi_2\left((gh)_*\rho_{A}(n)\right)\chi_2^{-1} \end{align*}
with $$\chi_2=r_g r_h r_{gh}^{-1} \in \Gamma_{A,k}.\qedhere$$
\end{proof}
%\begin{remark} The local-global morphism of \ref{karbonariusz} is $G/N$-equivariant.
%\end{remark}

\begin{remark}\label{uszczelka} The action of $G/N$ on $D_N$ has the following global analogue. Let $X$ be a projective smooth geometrically connected curve over $k$ with a faithful action of a group $G$ and let $N$ be a normal subgroup of $G$. As in the local case, one can consider the deformation functor $D_{X,N}$ of the pair $(X,N)$. A deformation of $(X,N)$ to $A$ is a triple $(X_A,\rho_A,i)$, where $X_A \to \Spec A$ is flat, $\rho_A \colon N \to \Aut_A X_A$ ia a lift of $\rho$ and $i \colon X \to X_A$ is an $N$-equivariant map inducing an isomorphism $X \to X_A \otimes_A k$. There is an action of the group $G/N$ on the functor $D_{X,N}$, defined as follows: 
An element $g \in G$ acts on $D_{X,N}(A)$ by mapping the triple $$\alpha=\big(X_A, N \xrightarrow{\rho_A} \Aut_A(X_A), X \xrightarrow{i} X_A\big)$$ to $$g_{*}\alpha=\big(X_A, N \xrightarrow{i_g} N \xrightarrow{\rho_A} \Aut_A(X_A), X \xrightarrow{\rho(g)} X \xrightarrow{i} X_A\big),$$ where $i_g$ is the inner conjugation $i_g(h)=ghg^{-1}$. Restricting the group action induces a map $\res \colon D_{X,G} \to D_{X,N}^{G/N}$.
\end{remark}

\begin{prop}\label{tanres} The action of $G/N$ on the tangent space $T_{D_N}=\HO^1(N,\der)$ corresponds to the standard action as defined in group cohomology. The restriction map induces on tangent spaces the map $$\res_{k[\varepsilon]} \colon T_{D_G} \to T_{D_N}^{G/N},$$ which is equal to the restriction morphism $\HO^1(G,\der) \to
\HO^1(N,\der)^{G/N}$ in group cohomology.\end{prop}
\begin{proof} 
Let $A=k[\varepsilon]$. We can write any element $\kappa\in D_G(A)$ as a class of a lift $\rho_A(g)=\gamma_g \rho(g)$ of $\rho$ to $A$ with $\gamma_g \in \der \cong \Gamma_{A,k}$ and $\kappa=[\gamma_g] \in \HO^1(G,\der)$. In the definiton of $g_{\ast}\kappa$ we can choose $r_g=\rho(g)$. By Lemma \ref{potwor}, we have  
\begin{eqnarray*}  g_{\ast}\rho_A(n) &=& \rho(g)\rho_A(g^{-1}ng)\rho(g)^{-1} \\
&=& \rho(g)\gamma_{g^{-1}ng}\rho(g^{-1}ng)\rho(g)^{-1}  \\ &=& \gamma_{g^{-1}ng}^g \rho(n). \end{eqnarray*} Thus on
cohomology the group action takes the form
$$g_*\big[\gamma_n\big]=\big[\gamma_{g^{-1}ng}^g\big].$$ This is the standard action (cf. \cite[p. 117]{Hochschild}). 

The fact that the restriction map on tangent spaces corresponds to the restriction map on cohomology is obvious.
\end{proof}

Restriction maps $D_G$ into $D_N^{G/N}$ -- a subfunctor of $D_N$ consisting of elements invariant under the action of $G/N$.
The information we have is already sufficient to compute universal deformation rings in some special cases.

\begin{prop}\label{purpura} If the functor $D\colon \kate \to \Set$ is pro-representable, say by a ring $R$, and a group $G$ acts on $D$, then
the functor $D^G$ is also pro-representable, and its universal ring is $$R/(\{gx-x \mid g \in G, x \in R\}).$$
\end{prop}
\begin{remark} The ring $S=R/(\{gx-x \mid g \in G, x \in R\})$ can be described in categorical terms as the ring of $\emph{co-invariants}$ of the action of $G$ on $R$, i.e., a ring universal for all morphisms $R \to S$ such that for any $g \in G$ the following diagram commutes: \begin{displaymath}\xymatrix{R \ar[rd] \ar[d]^g &\\R \ar[r]&S.}\end{displaymath} It is clear that to construct the ring of co-invariants, one has to take the quotient of $R$ by the ideal generated by elements of the form $gx-x$, $g\in G$, $x \in R$.\end{remark}
\begin{proof}[Proof of Proposition \ref{purpura}]
The action on the left of the group $G$ on $D\cong
h_R$ corresponds by Yoneda's lemma to the action of $G$ on the right on
the ring $R$. The functor $D^G$ is thus
pro-representable by the ring of co-invariants of the ring $R$ by the action of
$G$.
\end{proof}

\begin{wniosek} \mbox{} \begin{enumerate} \item[\textup{(i)}] If the functor $D_N$ is pro-representable, then so is $D_N^{G/N}$. \item[\textup{(ii)}] The tangent space to the functor $D_N^{G/N}$ is $\HO^1(N,\der)^{G/N}$.\qed\end{enumerate}\end{wniosek}

\begin{thm}\label{pyszczek} Assume that the order of $G/N$ is prime to $p$ 
and denote the versal deformation rings of functors $D_G$ and $D_N$ by
$R_G$ and $R_N$. Assume furthermore that the functor $D_N$ is pro-representable. Then
the restriction map $$\res\colon D_G \to D_N^{G/N}$$ is an isomorphism. In particular, $D_G$ is also pro-representable, and $$R_G \cong
R_N/(\{gx-x\mid g \in G, x \in R\}).$$\end{thm}
\begin{proof} We have seen in Proposition \ref{tanres} that on the tangent spaces the map $\res \colon D_G \to
D_N^{G/N}$ corresponds to the restriction map $\HO^1(G,\der) \to
\HO^1(N,\der)^{G/N}$. Since the order of $G/N$ is prime to $p$, in the Hochschild-Serre spectral sequence
$$\HO^p(G/N,\HO^q(N,\der)) \Rightarrow \HO^{n}(G,\der)$$ all elements outside the
zeroth column vanish. Thus the restriction maps $$\res \colon \HO^n(G,\der) \to \HO^n(N,\der)^{G/N}$$ are isomorphisms. For $n=1$, this shows that $\res_{k[\varepsilon]}$ is an isomorphism.

We  prove that the map $\res$ is smooth. Choose a small extension $$e \colon 0 \to I \to A' \to A \to 0$$ and
an element $$(\kappa_A,\lambda_{A'}) \in D_G(A)\times_{D_N^{G/N}(A)} D_N^{G/N}(A').$$ Let $\nu_e(\kappa_A) \in \HO^2(G,\der)\otimes I$ be the obstruction to lifting $\kappa_A$ to $A'$. Since $\res(\kappa_A)$ lifts to $A'$, we see that $\res(\nu_e(\kappa_A))=0$. The map $$\res \colon \HO^2(G,\der) \to \HO^2(N,\der)^{G/N}$$ is an isomorphism, and hence $\nu_e(\kappa_A)=0$. Thus, there is no obstruction to lifting $\kappa_A$ to $A'$ and there exists a $\kappa_{A'}\!\in
D_G(A')$ lying above $\kappa_A$. The deformation $\kappa_{A'}$
does not necessarily map to $\lambda_{A'}$, but its image $\res(\kappa_{A'})$ and $\lambda_{A'}$ do lie in the same fiber of $D_N^{G/N}(A')\to D_N^{G/N}(A)$. Since the functor $D_N^{G/N}$ is pro-representable, the fiber is a torsor under the action of $T_{D_N^{G/N}} \otimes I$. Hence we can define $\xi \in T_{D_N^{G/N}} \otimes I$ to be $$\xi=\res(\kappa_{A'})-\lambda_{A'}.$$ Let $\zeta \in T_{D_G} \otimes I$ be such that $\res(\zeta)=\xi$. Then the element $\kappa_{A'}-\zeta \in D_G(A')$ maps to $(\kappa_A,\lambda_{A'})$ by the map $$D_G(A')\to D_G(A)\times_{D_N^{G/N}(A)} D_N^{G/N}(A').$$ Hence $\res$ is \'etale, i.e., smooth and isomorphic on tangent spaces) and the claim follows from Proposition \ref{purpura}.\end{proof}

\begin{remark}\label{porzeczka} We follow the notation of Remark \ref{uszczelka}. For global deformation functors, the restriction map induces an isomorphism $\res \colon D_{X,G} \to D_{X,N}^{G/N}$ whenever $\mathrm{genus}(X) \geq 2$ (cf. \cite[p. 1909]{Maugeais}).\end{remark}

\begin{prop}
Let $\rho \colon G \to \Aut_k k\psl t \psr$ be a weakly ramified local action of a group $G$ whose order is divisible by $p$, but not divisible by $p^2$, $p \geq 3$. Then the functor $D_G$ is pro-representable by $W[\zeta_p+\zeta_p^{-1}]$ if $G=\mathbf{Z}/p$ or if $G$ is the dihedral group $D_p$. Otherwise, $D_G$ is pro-representable by $k$. 

\begin{proof} Consider the higher ramification groups $G_i$ of $G$ (cf. \cite[Ch. IV]{Serre}). Then $P=G_1$ is a normal $p$-Sylow subgroup of $G$ and hence is a cyclic subgroup of order $p$. We know by \S 5.3, Case 1 that $D_P$ is pro-representable by the ring $R_P=W[\zeta_p+\zeta_p^{-1}]$, where $\zeta_p$ is the $p$-th primitive root of unity. Thus $R_P$ is a complete discrete valuation ring. By Theorem \ref{pyszczek}, $D_G$ is pro-representable by the ring of $C$-co-invariants of $R_P$, where $C=G/P=\mathbf{Z}/m$. Since the order of $C$ is prime to $p$, $C$ has a trivial intersection with the wild inertia (cf. \cite[Corollaire 3 to Proposition IV.7]{Serre}). Thus we have the following two cases: \begin{enumerate}\item[(i)] The action of $C$ on $\mathfrak{m}_{R_P}/\mathfrak{m}_{R_P}^2$ is trivial, and thus $R_P/C=R_P$. \item[(ii)] The action of $C$ on $\mathfrak{m}_{R_P}/\mathfrak{m}_{R_P}^2$ is nontrivial, and thus $R_P/C=k$. \end{enumerate} The vector space $\mathfrak{m}_{R_P}/\mathfrak{m}_{R_P}^2$ is the tangent space $\HO^1(P,\der)$ to the functor $D_P$ and so we can distinguish between cases (i) and (ii) by a purely group cohomological calculation. Such a computation was done in \cite[\S 3.7]{Cornelissen}.
\end{proof}
 
\end{prop}

\end{section}

\begin{section}{Induction}\label{misja}
\begin{subsection}{Generalities}
Let $\rho \colon G \to \Gamma_k$ be a local $G$-action, and let $N$ be a normal subgroup of $G$. Consider the ring of $N$-invariants $k\psl t \psr^N$. 
Every automorphism $\rho(g)$ of $k\psl t \psr$ restricts to a map on $k\psl t \psr^N$. The map clearly depends only on the class $\overline{g}$, and  we denote it by $\rho(\overline{g})^{\sharp}$. We show below (Lemma \ref{prozelityzm}) that $k\psl t \psr^N=k\psl y \psr$ and hence $\rho^{\sharp} \colon G/N \to \Aut_k k\psl y\psr = \Gamma_k^{\sharp}$ is a local $G/N$-action. This allows us to consider the deformation functor $D_{G/N}$ of the local action $\rho^{\sharp}$. We will construct an induction map $\ind\colon D_G \to D_{G/N}$.

\begin{lemma}\label{prozelityzm} Let $\rho_A \colon G
\rightarrow \Gamma_A$ be a lift of $\rho$ to an object $A$ of $\kate$. Then
$$A\psl t\psr^G = A\psl y\psr ,$$ where $y=\prod_{g \in G}\rho_A(g)(t).$ \end{lemma}
\begin{proof} The proof will be by induction
on  length of $A$. 

\emph{Case $A=k$:} By Galois theory, the field of $G$-invariants of $k(\!(t)\!)$ is $k(\!(y)\!)$ (it is clearly invariant and it has the proper index.) The claim follows.

\emph{Case $A \neq k$:} In this case length of $A$ is at least $2$, and we can  find a principal surjection $\fhi \colon A \to A_0$ in
$\kate$ with kernel $\varepsilon A$ and $\maks_A \varepsilon = 0$. Furthermore, we may assume that the claim holds for $A_0$. We clearly
have $A\psl y \psr \subseteq A \psl t \psr^G$. Now, for any $f \in A
\psl t \psr^G$, the image $\fhi_*\! f$ of $f$ in $A_0\psl t \psr$ lies in $A_0\psl t \psr^G = A_0\psl \fhi_*\! y
\psr$. Choosing any lift $f_1$ of $\fhi_*\! f$ to  $A \psl y \psr$, we can write $f=f_1+h$ with $f_1 \in A \psl y \psr$, $h
\in \varepsilon A\psl t \psr$. Furthermore, $h$ has to be $G$-invariant, since so are both $f$ and $f_1$. We have $\varepsilon A\psl t \psr \simeq k \psl t \psr$ and
the induced action of $G$ on $k\psl t \psr$ is given by $\rho$. Since $h$ is $G$-invariant, it lies in
$k\psl \bar{y} \psr$, where $\bar{y}$ is the image of $y$ in $k \psl t \psr$. Consequently, we can write $h = \varepsilon u$ with $u \in A \psl y\psr$. This shows that $f= f_1+\varepsilon u$ lies in $A\psl y \psr$ and
finishes the proof of the inclusion $A\psl t \psr^G\subseteq A\psl y \psr$. %We will make use of the local flatness criterion (\cite{Matsumura},
%theorem 22.3). I claim that $A\psl t \psr^G$ is a flat $A$-algebra.
%Note first, that $A\psl Nt\psr \subseteq A\psl t \psr^G$ and hence
%$A\psl t \psr^G$ is a finite $A\psl Nt \psr$-module and is a
%noetherian ring. Thus, $A\psl t \psr^G$ is $\maks_A$-adically
%separated (\emph{cf.} a remark after the definition in \cite{Matsumura}, p.
%174). Thus by local flatness theorem one needs only to prove that
%$\mathrm{Tor}^A_1(k,A\psl t\psr^G)=0$.
\end{proof}

The induction map $D_G \colon D_{G/N}$ is constructed as follows: let $\kappa \in D_G(A)$ be a class of a lift $\rho_A \colon G \to \Aut_A A\psl t \psr$. Restricting the automorphisms to the subring $A\psl t \psr^N$ produces an action of $G/N$ on the ring $A\psl t \psr^N$. We define the local action $\rho_A^{\sharp}\colon G/N \to \Aut_A A\psl y \psr$ using the isomorphism $A\psl y \psr \simeq A \psl t \psr^N$ mapping $y$ to $\prod_{s\in N} \rho_A(s)(t)$. It is immediate to check that the equivalence class of $\rho_{A}^{\sharp}$ does not depend on the equivalence class of $\rho_A$ and that we obtain in this way a morphism of functors $D_G \to D_{G/N}$.

\begin{remark} We continue to use the notation and assumptions of Remark \ref{uszczelka}. The induction map also has a global analogue $\ind \colon D_{X,G} \to D_{X/N,G/N}$. It maps the triple $$\big(X_A, G \to \Aut_A(X_A), X \to X_A\big)$$ to $$\big(X_A/N, G/N \to \Aut_A(X_A/N), X/N \to X_A/N\big).$$ For details, see \cite{Mezard2}.
\end{remark}

\end{subsection}

\begin{subsection}{Tangent map}
In this subsection, we will compute the action of the induction map on the tangent spaces. Recall our setup: $\rho \colon G \to \Gamma_k$ is a local $G$-action and $N$ is a normal subgroup of $G$. Then we have extensions of discrete valuation rings: $$ k\psl t \psr^G \subseteq  k\psl t \psr^N \subseteq  k \psl t ]\!].$$ Lemma \ref{prozelityzm} gives $k \psl t \psr^G =k\psl y_G \psr$ and $k \psl t \psr^N=k \psl y_N \psr$, where $y_G=\prod_{g \in G} \rho(g)(t)$ and $y_N=\prod_{s \in N}\rho(s)(t)$. Recall that we write $\der = \Der_k k\psl t \psr$, and similarly we denote $\der^{\sharp}= \Der_k (k\psl t\psr^N) = \Der_k k\psl y_N\psr$. If $d\in\der^N$, then $d \rho(s) = \rho(s)d$ for $s \in N$, and hence $d(k\psl t \psr^N) \subseteq k \psl t \psr^N$. Restricting $d$ to the subring $k\psl t\psr^N$, we obtain an embedding $\der^N \hookrightarrow \der^{\sharp}$. The map takes the form $$a(t)\frac{d}{dt}\mapsto a(t)\frac{dy}{dt}\frac{d}{dy},$$ and hence is injective.
\begin{lemma} \label{rykliwy}
The maps \begin{align*} \der &\to  \mathcal{D}_{k\psl t \psr\!\slash k\psl t \psr^G}  &\derq&\to \mathcal{D}_{k\psl t \psr^N\!\!\slash k\psl t \psr^G} \\ d &\mapsto d(y_G) &d&\mapsto  d(y_G) \end{align*} are isomorphisms of $G$-modules. The maps agree with the natural embeddings $$\xymatrix{\der^N \ar@{^(->}[r] \ar[d]^{\wr} & \der^{\sharp} \ar[d]^{\wr}\\ \mathcal{D}_{k\psl t \psr/k\psl t \psr^G}\cap k\psl t \psr^N \ar@{^(->}[r] & \mathcal{D}_{k\psl t \psr^N/k\psl t \psr^G}.}$$
\end{lemma}
\begin{proof}
The maps are bijective by \cite[Corollaire III.6.2]{Serre}. The maps are $G$-equivariant by an easy calculation (cf. \cite[proof of Th\'{e}or\`{e}me 4.1.1]{BM}).\end{proof}Consider the ideal $\der_1 = \mathcal{D}_{k \psl t \psr ^N\!\!\slash k\psl t \psr^G} k \psl t \psr$ as a $G$-module. Clearly $$\der_1^N= \der_1 \cap \; k \psl t \psr^N = \derq.$$ We can picture relations between these ideals as follows: \begin{align*} &\der^N \subset \der\\ &\cap \quad \quad \cap \\ \derq =&\der_1^N  \subset \der_1 \\ &\cap \quad \quad \cap \\ k&\psl t \psr^N \!\subset k\psl t \psr.\end{align*} The map $\ind\colon D_G \to D_{G/N}$ induces on  tangent spaces the map $$\ind_{k[\varepsilon]} \colon \HO^1(G,\der) \to \HO^1(G/N,\der^{\sharp}).$$

\begin{thm}\label{przebaczenie} The following diagram is commutative: 
$$\xymatrix{\HO^1(G/N,\der^N)
\ar@{^(->}[d]^{\inf} \ar[r] &\HO^1(G/N,\derq) \ar@{^(->}[d]^{\inf}\\
\HO^1(G,\der) \ar[ur]_{\ind_{k[\varepsilon]}} \ar[r]& \HO^1(G,\derj),}$$ where $\inf$ is the usual inflation map in group cohomology.\end{thm}
\begin{proof} The square commutes by functoriality of the maps $\inf$. Since the right most map is injective, it is enough to prove that the right hand side triangle commutes. Denote $A=k[\varepsilon]$. A lift of $\rho_A\colon G \to \Gamma_A$ corresponding to a 1-cocycle $d\colon G \to \der$ is given by $$\rho_A(g)(x)=\rho(g)(x) + \varepsilon d_g(\rho(g)(x)).$$ The image
of $[d_g]$ by the induction map is the class of a cocycle
$(d^{\sharp}_{\bar{g}})$ such that
$$\rho^{\sharp}_A(\bar{g})(u)=\rho^{\sharp}(\bar{g})(u) + \varepsilon
d_{\bar{g}}^{\sharp}\left(\rho^{\sharp}(\bar{g})(u)\right).$$ Denote $y_N^A=\prod_{s \in N} \rho_A(s)(t)$. Note that \begin{equation}\label{Gombrowicz}y_N^A=\prod_{s \in N} \rho_A(s)(t) = \prod_{s \in N} \left(f_s + \varepsilon d_s (f_s)\right)=y_N+\varepsilon y_N \sum_{s \in N} \frac{d_s(f_s)}{f_s}.\end{equation} On the one hand, we have
\begin{eqnarray}\nonumber \rho^{\sharp}_A(\bar{g})(y_N^A) &=& \rho_A(g)(y_N^A)=\prod_{s \in N} \rho_A(gs)(t) = \prod_{s \in N}
\left( f_{gs} + \varepsilon d_{gs}(f_{gs}) \right) \\ \label{Witkiewicz} &=& y_N^g
+ \varepsilon y_N^g \sum_{s \in N} \frac{d_{gs}(f_{gs})}{f_{gs}}.\end{eqnarray} On the other hand,
\begin{equation}\label{Milosz}\rho_A^{\sharp}(\bar{g})(y_N^A) = \rho^{\sharp}(\bar{g})(y_N^A) +
\varepsilon d_{\bar{g}}^{\sharp}(y_N^g),\end{equation} where
$\rho^{\sharp}(\bar{g})(y_N^A)$ denotes the power series $\rho^{\sharp}(\bar{g})(y_N)$
with $y_N^A$ substituted for $y_N$. By equation (\ref{Gombrowicz}), we can compute
\begin{eqnarray}\nonumber \rho^{\sharp}(\bar{g})(y_N^A) &=& \rho^{\sharp}(\bar{g})(y_N) + \frac{d\rho^{\sharp}(\bar{g})(y_N)}{dy_N}(y_N^A-y_N)
\\\label{Schulz} &=& y_N^g+\varepsilon y_N\frac{dy_N^g}{dy_N}\sum_{s \in N} \frac{d_s(f_{s})}{f_s}.\end{eqnarray} Since $d_{\bar{g}}^{\sharp}(y_N^g) = \frac{dy_N^g}{dy_N} d_{\bar{g}}^{\sharp}(y_N)$, equations (\ref{Witkiewicz}), (\ref{Milosz}) and (\ref{Schulz}) give
\begin{equation}\label{Herbert}d_{\bar{g}}^{\sharp}(y_N)=y_N^g \left(\frac{dy_N^g}{dy_N}\right)^{-1}\!\!\sum_{s \in N}  \frac{d_{gs}(f_{gs})}{f_{gs}} - y_N \sum_{s \in N}  \frac{d_s( f_s)}{f_s}.\end{equation}\renewcommand{\qedsymbol}{}To see that the map given by this formula makes the diagram commute, we need a computation. We break the proof here in order to state a number of lemmas necessary to perform this computation. \end{proof}Since we will need to make similar computations later on, we state the results in a version useful also for those future applications.

\begin{lemma}\label{nosorozec} Let $\rho \colon G \to \Gamma_k$ be a local $G$-action, $N$ be a normal subgroup of $G$ and denote $f_g=\rho(g)(t)$. Let $a \colon G \to k(\!(t)\!)$ be a cochain such that $$a_{gs}=a_s^g+a_g, \quad g\in G, s \in N.$$ Let $$\alpha = -  y_N \left(\frac{dy_N}{dt}\right)^{-1}\!\! \sum_{s \in N} a_s\frac{df_s}{dt}f_s^{-1} \in k (\!(t)\!).$$ %$$\alpha^g-\alpha=a_g-\left(\frac{dy_N}{dt}\right)^{-1}\!\!\left(y_N^g \left(\frac{dy_N^g}{dy_N}\right)^{-1}\!\!\sum_{s \in N}  a_{gs}\frac{df_{gs}}{dt}f_{gs}^{-1} - y_N \sum_{s \in N}  a_s\frac{df_s}{dt}f_s^{-1}\right).$$ 
Then \begin{enumerate} \item[\textup{(i)}] $\displaystyle \alpha^g-\alpha=a_g-y_N^g\left(\frac{dy_N^g}{dt}\right)^{-1}\!\!\sum_{s \in N}  a_{gs}\frac{df_{gs}}{dt}f_{gs}^{-1}+ y_N\left(\frac{dy_N}{dt}\right)^{-1}\!\!\sum_{s \in N}  a_s\frac{df_s}{dt}f_s^{-1}.$
In particular, $$\alpha^s-\alpha=a_s \quad \text{for} \quad s\in N.$$
\item[\textup{(ii)}] Denote by $\ord$ the $t$-adic valuation on $k(\!(t)\!)$ and by $d = \ord\left(\mathcal{D}_{R/R^N}\right)$  the order of the different of the extension $R^N \subseteq R$. Then $$\ord(\alpha) \geq \inf_{s \in N} \ord(a_s) + \# N - d - 1.$$ \end{enumerate}
\end{lemma}  \begin{proof} (i) We have $y_N^g=\prod_{s\in N} f_{gs}$ and thus \begin{equation}\label{mikrostruktura}\frac{dy_N^g}{dt}=y_N^g\sum_{s\in N}\frac{df_{gs}}{dt}f_{gs}^{-1}.\end{equation}
We compute \begin{eqnarray}\label{ciernisty}\alpha^g&=&-y_N^g\left(\frac{d y_N^g}{df_g}\right)^{-1}\!\!\sum_{s\in N} a_s^g\frac{df_{gs}}{df_g}f_{gs}^{-1} \\ \nonumber &=& y_N^g\left(\frac{d y_N^g}{dt}\right)^{-1}\!\!a_g\sum_{s\in N}\frac{df_{gs}}{dt}f_{gs}^{-1} -y_N^g\left(\frac{d y_N^g}{dt}\right)^{-1}\!\!\sum_{s\in N}a_{gs}\frac{df_{gs}}{dt}f_{gs}^{-1}.\end{eqnarray} Equations (\ref{mikrostruktura}) and (\ref{ciernisty}) yield \begin{eqnarray*}\alpha^g&=&a_g-y_N^g\left(\frac{d y_N^g}{dt}\right)^{-1}\!\!\sum_{s\in N}a_{gs}\frac{df_{gs}}{dt}f_{gs}^{-1}.\end{eqnarray*} This gives the claim. 

 (ii) We know that $\ord(\frac{dy_N}{dt}) =d$ \cite[Corollaire III.6.2]{Serre} and $\ord(y_N)=\#N$. We now have
$$\ord(\alpha) \geq \ord(y_N)-\ord\left(\frac{dy_N}{dt}\right)+\inf_{s\in N}{\ord(a_s \frac{df_s}{dt}f_s^{-1})} = \# N - d + \inf_{s \in N}(\ord(a_s))-1, $$ since $\ord(f_s)=1$ and $\ord(\frac{df_s}{dt})=0$.
\end{proof}

\begin{prop}[Effective version of Hilbert's theorem 90]\label{przesiedlenie} Let $I \subseteq J$ be two fractional ideals of $R=k \psl t \psr$ and let $\rho\colon G\to \Aut_k R$ be a local $G$-action. Write $d = \ord(\mathcal{D}_{k\psl t \psr/k\psl t \psr^N})$. Assume that $$\ord (IJ^{-1}) \geq d - \# G +1.$$ Then the natural map $\HO^1(G,I) \to \HO^1(G,J)$ is the zero map.
\end{prop}
\begin{proof}[Proof of Proposition \ref{przesiedlenie}]
Apply Lemma \ref{nosorozec} to the cocycle $a\colon G \to I$ and $N=G$. We obtain in this way an $\alpha \in k(\!(t)\!)$ such that $\alpha^g-\alpha=a_g$ and
$$\ord(\alpha) \geq \inf_{g \in G}(\ord(a_g))+\# G-d-1.$$ By the assumption, this gives $$\ord(\alpha) \geq \ord(I)-\ord(IJ^{-1})=\ord(J).$$ Hence $\alpha \in J$. Since $a_g = \alpha^g - \alpha$, the image of $(a_g)$ vanishes in $\HO^1(G,J)$.
\end{proof}

\begin{remark}
Denote by $K$ the fraction field of $R$. By the additive Hilbert's theorem 90, the group $\HO^1(G,K)$ vanishes and hence the composite map $\HO^1(G,I) \to \HO^1(G,J) \to \HO^1(G,K)$ is zero. The lemma gives an explicit bound on where the ``splitting" happens.
\end{remark}

We will now give a version of Lemma \ref{nosorozec} in a slightly different context.

\begin{wniosek}\label{prefektura} Let $d \colon G \to \der$ be a cochain such that $$d_{gs}=d_s^g+d_g, \quad g\in G, s \in N.$$ Let $$\alpha = - y_N  \frac{dy_G}{dy_N} \sum_{s \in N} \frac{d_s(f_s)}{f_s}.$$ Then \begin{enumerate} \item[\textup{(i)}] $\displaystyle \alpha^g-\alpha=d_g(y_G)-y_N^g\frac{dy_G}{dy_N}\left(\frac{dy_N^g}{dy_N}\right)^{-1}\!\!\sum_{s \in N}  \frac{d_{gs}(f_{gs})}{f_{gs}} + y_N\frac{dy_G}{dy_N} \sum_{s \in N}  \frac{d_s( f_s)}{f_s}.$ \item[\textup{(ii)}] $\alpha$ lies in $\der_1$ (under the identification of Lemma \ref{rykliwy}). \end{enumerate}
\end{wniosek} 
\begin{proof} (i) Apply Lemma \ref{nosorozec} to the cocycle $a \colon G \to k(\!(t)\!)$ obtained by composing $d \colon G \to \der$ with the isomorphism $\der \simeq \mathcal{D}_{k\psl t\psr/k\psl t \psr^G} \subseteq k(\!(t)\!)$. Then $$a_g=d_g(y_G)=\frac{dy_G}{dt}d_g(t)$$ and the claim follows. 

(ii) From the definition of $\alpha$, we get $$\ord(\alpha)\geq \#N+\ord\left(\frac{dy_G}{dy_N}\right)-1\geq \ord\left(\frac{dy_G}{dy_N}\right).$$ By Lemma \textup{\ref{rykliwy}}, $\ord(\frac{dy_G}{dy_N})=\ord(\der_1)$, which gives the claim.
\end{proof}

\begin{proof}[Proof of Theorem \ref{przebaczenie} continued]
We will interpret elements of $\der$ and $\derq$ not as derivations, but as elements of the differents (cf. Lemma \ref{rykliwy}). Corollary \ref{prefektura}, together with equation (\ref{Herbert}) shows that there exists an element $\alpha \in \der_1$ such that $$\alpha^g-\alpha=d_g(y_G)-\frac{dy_G}{dy_N}d_{\bar{g}}^{\sharp}(y_N)= d_g(y_G)-d_{\bar{g}}^{\sharp}(y_G).$$ 
Since $$\ind_{k[\varepsilon]}([d_g(y_G)])=[d_{\bar{g}}^{\sharp}(y_G)]=[d_g(y_G)-\alpha^g+\alpha],$$ this proves that $$\inf(\ind_{k[\varepsilon]}([d_g(y_G)]))=[d_g(y_G)],$$ and hence the right hand side triangle commutes. 
\end{proof}
\end{subsection}
\end{section}

\begin{section}{D\'evissage}
\begin{subsection}{Tangent map}

\begin{thm}\label{Mandelsztam} There is a morphism $\gamma$ (constructed below) such that the following diagram commutes:

$$\xymatrix{\HO^1(G/N,\der^N) \ar@{^{(}->}[r]^-{\infl} \ar@{=}[d] & \HO^1(G,\der) \ar[r]^-{\res} \ar[d]^-{\ind}&\HO^1(N,\der)^{G/N}  \ar[r]^-{\transgresja} \ar@{..>}[d]^{\gamma}&\HO^2(G/N,\der^N)  \ar[r]^-{\infl} \ar@{=}[d]&\HO^2(G,\der) 
 \\\HO^1(G/N,\der^N) \ar[r] & \HO^1(G/N,\derq)  \ar[r]^-{\sigma}&\HO^1(G/N,\der^{\sharp}\!/\der^N)  \ar[r]^-{\partial}&\HO^2(G/N,\der^N). &
}$$
\end{thm}
\begin{proof}
The map $\transgresja$ is the trangression and the upper row of this diagram is the extended inflation-restriction sequence coming from the Hochschild-Serre spectral sequence in group cohomology \cite[Remark following Theorem III.2]{Hochschild}. The bottom row arises from a long exact sequence in cohomology. The short exact sequence $$ 0 \rightarrow \der \rightarrow \der_1 \rightarrow \der_1\!/\der \rightarrow 0 $$ induces a long exact sequence of cohomology groups $$ 0 \rightarrow \der^N \rightarrow \der_1^N \rightarrow (\der_1\!/\der)^N \rightarrow \HO^1(N,\der) \rightarrow \HO^1(N,\der_1) \rightarrow \ldots.$$ Since the map $\HO^1(N,\der) \to \HO^1(N,\der_1)$ is zero (Proposition \ref{przesiedlenie}), we get a short exact sequence $$ 0 \rightarrow \der_1^N\!/\der^N \rightarrow (\der_1\!/\der)^N \rightarrow\HO^1(N,\der) \rightarrow 0.$$ Again, it induces a long exact sequence in cohomology \begin{align}\label{Gaza}\xymatrix{0 \ar[r] & (\der_1^N\!/\der^N)^{G/N} \ar[r] & (\der_1\!/\der)^G \ar[r] & \HO^1(N,\der)^{G/N}  \ar[r]^<<<<<{-\gamma} &} \\ \nonumber \xymatrix{ \ar[r]^<<<<<{-\gamma} & \HO^1(G/N, \der_1^N\!/\der^N) \ar[r] & \HO^1(G/N,(\der_1\!/\der)^N) \ar[r] & \ldots,}\end{align} where $-\gamma \colon \HO^1(N, \der)^{G/N} \rightarrow \HO^1(G/N,\der_1^N\!/\der^N)$ is the boundary map. We take $\gamma$ to be $\gamma=-(-\gamma)$, i.e., $\gamma$ is to take values opposite to the boundary map. 

We now analyse the construction step by step in order to obtain an explicit definition of the map $\gamma$. For a cocycle $d \colon N  \rightarrow \der$, we can find an element $\alpha \in \der_1$ such that $d_s=\alpha^s-\alpha$ for $s \in N$. If, furthermore, $[d_s]$ is $G/N$-invariant, then the cocycles $(d_s)$ and $(d^g_{g^{-1}sg})$ differ by a coboundary, i.e., there exists a $b_g \in \der$ such that \begin{equation}\label{zagrozenie}d^g_{g^{-1}sg}=d_s + b_g^s - b_g.\end{equation} This shows that $\alpha^g-\alpha-b_g$ is $N$-invariant and hence $\alpha^g-\alpha$ lies in $\der_1^N+\der$ and we can write $$\alpha^g-\alpha=a_g+b_g \quad \text{with} \quad a_g \in \der_1^N, b_g \in \der.$$ The elements $a_g$ and $b_g$ are well-defined only up to an element of $\der^N$. Since $$a_{gs}+b_{gs}=(\alpha^s-\alpha)^g+(\alpha^g-\alpha)=a_s^g+b_s^g+a_g+b_g,$$ and since $a_s \equiv 0 \pmod {\der^N}$, we see that the class of $a_g$ in $\der_1^N\!/\der^N$ depends only on the class of $g$ in $G/N$. Furthermore, by the construction of the boundary map in (\ref{Gaza}) we get $\gamma([d_s])=-[\overline{a}_{\bar{g}}]$.

If the cocycle $(d_s)$ is a restriction of a cocycle $d\colon G \to \der$, then in (\ref{zagrozenie}) we can take $b_g=d_g$. We then get $$\gamma([d_s])=[d_g-\alpha^g+\alpha].$$ This shows that the central square commutes.

By \cite[Proposition 1.6.5, and beginning of its proof]{Neukirch} we know that we can choose $b_g \in \der$ in (\ref{zagrozenie}) in such a way that \begin{enumerate} \item[(i)] $b_s=d_s$ for $s\in N$, \item[(ii)] $b_{sg}=b_g^s+b_s$ for $s\in N$ and $g \in G$, \item[(iii)] $b_{gs}=b_s^g+b_g$ for $s\in N$ and $g \in G$. \end{enumerate} Then $\transgresja([d_s])= [b_h^g-b_{gh}+b_g]$ (cf. ibid.). On the other hand $\gamma([d_s])=-[\overline{a}_{\bar{g}}].$ Since the map $\HO^1(G/N,\der^{\sharp}\!/\der^N) \rightarrow \HO^2(G/N,\der^N)$ is the boundary map in cohomology, and since $a_g+b_g=\alpha^g-\alpha$ is a coboundary, we see that the third square also commutes. \end{proof} For future reference, we state explicitly the form of the map $\gamma$.
\begin{wniosek}\label{architektura} In the notation of the proof of Theorem \ref{Mandelsztam}, we have $$\gamma([d_s])=-[\overline{a}_{\bar{g}}].$$ \end{wniosek}

We will now investigate the morphism of functors $$\Phi = (\res,\ind) \colon D_G \to D_N^{G/N} \times D_{G/N}.$$\begin{wniosek}\label{Jugendstill} We have: \begin{enumerate} \item[\textup{(i)}] $\ker \Phi_{k[\varepsilon]} = \coker(\der_1^{G} \to (\der_1^N\!/\der^N)^{G/N})$; \item[\textup{(ii)}] $\im \Phi_{k[\varepsilon]} = \HO^1(N,\der)^{G/N} \times_{\HO^1(G/N, \der^{\sharp}\!/\der^N)} \HO^1(G/N,\der^{\sharp})$. \end{enumerate} 
\end{wniosek}
\begin{proof} (i) From Theorem \ref{Mandelsztam} we see that $$\ker \Phi_{k[\varepsilon]}= \ker \left(\HO^1(G/N,\der^N)\to\HO^1(G/N,\der_1^N)\right).$$ The lower row in the diagram in  Theorem \ref{Mandelsztam} is a part of a long exact sequence associated to $$0 \to \der^N \to \der_1^N \to \der_1^N/\der^N \to 0.$$ This proves that  $$\ker \left(\HO^1(G/N,\der^N)\to\HO^1(G/N,\der_1^N)\right) = \coker \left(\der_1^G \to (\der_1^N\!/\der^N)^{G/N}\right).$$
(ii) The inclusion $$\im \Phi_{k[\varepsilon]} \subseteq \HO^1(N,\der)^{G/N} \times_{\HO^1(G/N, \der^{\sharp}\!/\der^N)} \HO^1(G/N,\der^{\sharp})$$ is trivial. 
For the other one, choose $$(\kappa,\lambda)\in\HO^1(N,\der)^{G/N} \times_{\HO^1(G/N, \der^{\sharp}\!/\der^N)} \HO^1(G/N,\der^{\sharp}).$$
Then by Theorem \ref{Mandelsztam}, $$\transgresja(\kappa)=\partial(\gamma(\kappa))=\partial(\sigma(\lambda))=0,$$ and hence there is a $\zeta \in \HO^1(G,\der)$ such that $\res(\zeta)=\kappa$.
We have then $\sigma(\ind(\zeta))=\sigma(\lambda)$, and hence $\ind(\zeta) - \lambda \in \ker \sigma$. Hence there is a $\xi \in \HO^1(G/N,\der^N)$ such that $$\ind(\zeta) - \lambda = \ind (\inf (\xi)).$$
Choosing $\zeta'=\zeta-\inf(\xi)$, we get $\Phi(\zeta')=(\kappa,\lambda)$. This proves the opposite inclusion.
\end{proof}
\end{subsection}

\begin{subsection}{Main result}

Inspired by Corollary \ref{Jugendstill}.ii, we pose the following question.

\begin{question} Does there exist a pro-representable functor $F$ with the tangent space $\HO^1(G/N, \der^{\sharp}\!/\der^N)$ and with morphisms $D_N^{G/N} \to F$ and $D_{G/N} \to F$ such that the morphism $\Phi$ maps $D_G$ into the fibered product $D_N^{G/N} \times_F D_{G/N}$ and such that the induced  morphism $$D_G \to D_N^{G/N} \times_F D_{G/N}$$ is smooth? \end{question}

\begin{remark}\label{pazdziernik} If the answer to this question is positive, and if the functors $D_N$, $D_{G/N}$ and $F$ are pro-representable with universal deformation rings $R_N$, $R_{G/N}$ and $R_F$, respectively, then the versal deformation ring $R_G$ of $D_G$ is of the form  $$R_G= \left(((R_N)/(G/N)) \widehat{\otimes}_{R_F} R_{G/N} \right) \psl x_1, x_2, \ldots, x_s\psr,$$ where $R_N/(G/N)$ denotes the ring of co-invariants (cf. Proposition \ref{purpura}) and $\widehat{\otimes}$ denotes the completed tensor product. \end{remark}

We begin be recalling the definition of an obstruction space in a relative setting, when a functor is replaced by a morphism of functors (for a general study of these morphism, cf.\ \cite{Fantechi}).

\begin{dfn}\label{sentyment} Let $f \colon D \to E$ be a morphism of functors $D,E \colon \kate \to \Set$. An obstruction space\index{obstruction space!to a morphism of functors} to $f$ is a vector space $V$ over $k$ together with a collection of maps $$\nu_e \colon D(A) \times_{E(A)} E(A') \to V \otimes I$$ defined for every small extension $e \colon 0 \to I \to A' \to A \to 0$ in $\kate$ (i.e., a small surjection $A' \to A$ with kernel $I$) and satisfying the following condition: For a morphism $e \to e'$ of small extensions \begin{displaymath}\xymatrix{ 0 \ar[r] &I \ar[d] \ar[r] & A' \ar[r]\ar[d] & A \ar[r]\ar[d] & 0 \\  0 \ar[r] &J  \ar[r] & B' \ar[r] & B \ar[r] &0 ,}\end{displaymath} the induced diagram \begin{displaymath} \xymatrix{ D(A) \times_{E(A)} E(A') \ar[r]^<<<<<{\nu_e} \ar[d] & V \otimes I\ar[d]\\ D(B) \times_{E(B)} E(B') \ar[r]^<<<<<{\nu_{e'}}& V \otimes J}\end{displaymath}commutes.\end{dfn} 
It is clear that if $\xi \in  D(A) \times_{E(A)} E(A')$ lies in the image of the map $$D(A') \to  D(A)\times_{E(A)}E(A'),$$ then $\nu_e(\xi)=0$. 

\begin{dfn} An obstruction space $V$ to a morphism $f \colon D \to E$ of functors $D,E\colon \kate \to \Set$ is called complete\index{obstruction space!complete} if for every small extension $e$ the following condition holds: If $\xi \in D(A) \times_{E(A)} E(A')$ is such that $\nu_e(\xi)=0$, then $\xi$ lies in the image of the map $$D(A') \to  D(A)\times_{E(A)}E(A').$$ \end{dfn}

For $E$ being the one-point functor, we recover the definition of an obstruction space to $D$.

\begin{prop}\label{wiarygodnosc} Let $D,D_1,D_2,F\colon \kate \to \Set$ be functors with a commutative diagram of morphism $$\xymatrix{ D \ar[r] \ar[d]& D_1 \ar[d]^f \\ D_2 \ar[r]^g&F.}$$ Assume that $F$ is pro-representable and that the morphism $D \to D_1 \times_F D_2$ is smooth. Then the tangent space $T_F$ to the functor $F$ is an obstruction space to the morphism $\Phi \colon D \to D_1 \times D_2$ (in the sense of definition \ref{sentyment}).\end{prop}
\begin{proof} Let $e \colon 0 \to I \to A' \to A \to 0$ be a small extension. 
We should construct a map $$\nu_e^{\Phi} \colon D(A) \times_{(D_1(A) \times D_2(A))} \left(D_1(A') \times D_2(A')\right) \to T_F \otimes I.$$ 
Given an element $$\left(\kappa_A,(\kappa_{A'}^1,\kappa_{A'}^2)\right) \in D(A) \times_{(D_1(A) \times D_2(A))} \left(D_1(A') \times D_2(A')\right),$$ the elements $f_{A'}(\kappa^1_{A'})$ and $g_{A'}(\kappa_{A'}^2)$ lie in the same fiber of $F(A') \to F(A)$. Since $F$ is pro-representable, the fibers are torsors under the action of $T_F \otimes I$.
Thus we can define $$\nu^{\Phi}_e(\left(\kappa_A,(\kappa_{A'}^1,\kappa_{A'}^2)\right))= g_{A'}(\kappa_{A'}^2) - f_{A'}(\kappa^1_{A'}) \in T_F \otimes I.$$ 
This provides $T_F$ with a structure of an obstruction space. 
If the obstruction vanishes, then $(\kappa_{A'}^1,\kappa_{A'}^2) \in D_1(A')\times_{F(A')}D_2(A')$, and by smoothness of the map $D \to D_1\times_F D_2$ we see that $\left(\kappa_A,(\kappa_{A'}^1,\kappa_{A'}^2)\right)$ lies in the image of the map \[ D(A') \to D(A) \times_{(D_1(A) \times D_2(A))} \left(D_1(A') \times D_2(A')\right).\qedhere \]\end{proof}

In the next theorem, we independently verify that the conclusion of Remark \ref{wiarygodnosc} is true for the functors $D=D_G$, $D_1=D_N^{G/N}$ and $D_2=D_{G/N}$. This provides some evidence for plausabilty of existence of the functor $F$.

\begin{thm}\label{szczezl} Assume that the functors $D_N$ and $D_{G/N}$ are pro-representable. Then the vector space $\HO^1(G/N,\der^{\sharp}\!/\der^N)$ is a complete obstruction space to the morphism $\Phi = (\res,\ind) \colon D_G \to D_N^{G/N} \times D_{G/N}.$
\end{thm}
\begin{proof} The proof is quite long, and we will break it in several steps. A part of the proof is isolated in Lemma \ref{nieustepliwosc}, stated and proven only after the main proof. Lemma \ref{nieustepliwosc} needs in turn a certain computation in the Hochshild-Serre spectral sequence, given in Theorem \ref{nemezis} in the appendix.

Let $e \colon 0 \to I \to A' \to A \to 0$  be a small extension in $\kate$. We will construct a map $$\nu_e^{\Phi} \colon D_G(A) \times_{(D_N^{G/N}(A) \times D_{G/N}(A))} \left(D_N^{G/N}(A') \times D_{G/N}(A')\right) \to \HO^1(G/N,\der^{\sharp}\!/\der^N) \otimes I.$$ Choose an element $$\left(\kappa_A,(\kappa_{A'}^N,\kappa_{A'}^{G/N})\right) \in D_G(A) \times_{(D_N^{G/N}(A) \times D_{G/N}(A))} \left(D_N^{G/N}(A') \times D_{G/N}(A')\right).$$ The deformation $\kappa_A\in D_G(A)$ defines an obstruction class $$\underline{obs}=\nu_e(\kappa_A)\in\HO^2(G,\der)\otimes I.$$

\emph{The class $\underline{obs}_0$\textup{:}} By the spectral sequence in group cohomology $$E_2^{pq}=\HO^p(G/N,\HO^q(N,\der)) \Rightarrow \HO^n(G,\der),$$ we know that there exists a filtration $0=E^0 \subseteq E^1 \subseteq E^2 \subseteq E^3=\HO^2(G,\der)$ such that $E^{k+1}/E^k \simeq E_{\infty}^{k,2-k}$. By Lemma \ref{nieustepliwosc} proven below, we see that from the fact that $\res(\kappa_A) \in D_N^{G/N}(A)$ lifts to $A'$ we can conclude that the images of $\underline{obs}$ in $\HO^2(N,\der)\otimes I$ and $\HO^1(G/N,\HO^1(N,\der))\otimes I$ vanish. This shows that the images of $\underline{obs}$ in $(E^3/E^2)\otimes I$ and $(E^2/E^1) \otimes I$ are trivial, thus $\underline{obs} \in E^1 \otimes I$, and hence there exists an $\underline{obs}_0 \in \HO^2(G/N,\der^N) \otimes I$ such that $\underline{obs}=\inf(\underline{obs}_0)$. 

Just like for restriction, we also know that $\ind(\kappa_A)\in D_{G/N}(A)$ lifts to $A'$. We will now show that the obstruction $\nu_e(\ind(\kappa_A))$ to lifting $\ind(\kappa_A) \in D_{G/N}(A)$ to $A'$ is equal to the image of $\underline{obs}_0$ in $\HO^2(G/N,\der^{\sharp}) \otimes I$, and from this conclude that \begin{equation}\label{przedprzedostatni}\underline{obs}_0 \in \ker \left( \HO^2(G/N,\der^N)\to \HO^2(G/N,\der^{\sharp}) \right)\otimes I.\end{equation}

\emph{The element $\ind(\kappa_A) \in D_{G/N}(A)$ is the image of $\underline{obs}_0$\textup{:}} Choose a representative $\rho_A\colon G \to \Gamma_A$ of $\kappa_A$.
Write $\underline{obs}_0$ as the class $\underline{obs}_0=[\eta(\bar{g},\bar{h})]$ of a 2-cocycle $$\eta \colon G/N \times G/N \to \Gamma_{A',A} \simeq \der^{\sharp}\otimes I.$$ We can choose the cocycle in such a way that $\eta(\id,\id)=0$. 
Since $\underline{obs} = \inf(\underline{obs}_0)$, there exists a set-theoretic lift $\rho^{\ast}_{A'} \colon G \to \Gamma_{A'}$ of $\rho_A$ to $A'$ such that $$\rho^{\ast}_{A'}(g)\rho^{\ast}_{A'}(h)=\eta(\bar{g},\bar{h})\rho^{\ast}_{A'}(gh)$$ (The obstruction to lifting $\rho_A$ to $A'$ is the class of $\eta$; we can always modify the choice of $\rho^{*}_{A'}$ so that it induces a given cocycle in this class.) For such a choice of a lift, the restriction of $\rho^{\ast}_{A'}$ to $N$ gives a group homomorphism $N \to \Gamma_{A'}$.
We can thus consider the ring of $N$-invariants $A'\psl t \psr^N = A \psl y_{A'} \psr$ (cf. Lemma \ref{prozelityzm}). Since we have $$\rho^{\ast}_{A'}(s)\rho^{\ast}_{A'}(g)=\rho^{\ast}_{A'}(sg)=\rho^{\ast}_{A'}(g)\rho^{\ast}_{A'}(g^{-1}sg) \quad \text{for} \quad g \in G, s\in N,$$ we see that the morphism $\rho_{A'}^{\ast}$ preserves $A'\psl t \psr^N$ and the morphism it restricts to on $A'\psl t \psr^N$ depends only on the class $\bar{g}$ of $g$ in $G/N$.
We denote the induced map by $\rho^{\ast}_{A'}(\bar{g})^{\sharp}\in\Gamma_{A'}^{\sharp}$.  We see that $\eta(\bar{g},\bar{h})=\rho^{\ast}_{A'}(g)\rho^{\ast}_{A'}(h)\rho^{\ast}_{A'}(gh)^{-1}$ also preserves $A'\psl t \psr^N$. Furthermore, $$\rho^{\ast}_{A'}(\bar{g})^{\sharp}\rho^{\ast}_{A'}(\bar{h})^{\sharp}=\eta(\bar{g},\bar{h})^{\sharp}\rho^{\ast}_{A'}(\overline{gh})^{\sharp},$$ where $\eta(\bar{g},\bar{h})^{\sharp}$ is the restriction of $\eta(\bar{g},\bar{h})$ to  $A'\psl t \psr^N$. Thus, the obstruction class $\nu_e^{G/N}(\ind (\kappa_A))$ to lifting $\ind(\kappa_A)$ to $A'$ is equal the image of $\underline{obs}_0$ in $\HO^2(G/N,\der^{\sharp})\otimes I$. 

\emph{The class $\underline{Obs}$\textup{:}} By (\ref{przedprzedostatni}), we can conclude from Theorem \ref{Mandelsztam} that there exists a class $$\lambda \in \HO^1(G/N,\der^{\sharp}\!/\der^N) \otimes I$$ such that $\underline{obs}_0=\partial\lambda$. Choose a cochain $$u \colon G/N \to \der^{\sharp}\otimes I$$ such that the composite map $$\bar{u} \colon G/N \xrightarrow{u}\der^{\sharp}\otimes I \to (\der^{\sharp}\!/\der^N)\otimes I$$ is a cocycle and $\lambda$ is the class of $\bar{u}$. If necessary, we can modify $u$ in order to get $u(\id)=0$.

Recall that we have chosen a lift $\rho^{\ast}_{A'}$ of $\rho_A$ to $A'$ such that $$\rho^{\ast}_{A'}(g)\rho^{\ast}_{A'}(h)=\eta(\bar{g},\bar{h})\rho^{\ast}_{A'}(gh)$$ and that $\underline{obs}_0=[\eta(\bar{g},\bar{h})]$. We can furthermore assume that $\eta$ has been chosen in such a way that $$\eta(\bar{g},\bar{h})=u(\bar{h})^{\bar{g}}-u(\overline{gh})+u(\bar{g}).$$
Recall that the restriction of $\rho^{\ast}_{A'}$ to $N$ is a group homomorphism, and hence both $[\rho^{\ast}_{A'}]$ and $\kappa^N_{A'}$ are in the fiber of $D_N^{G/N}(A')\to D_N^{G/N}(A)$ lying over $\kappa_A^N$. Since $D_N^{G/N}$ is pro-representable, it makes sense to define $$\alpha = \kappa^N_{A'}-[\rho^{\ast}_{A'}] \in \HO^1(N,\der)^{G/N} \otimes I.$$

The elements $r_{\bar{g}} \in \der^{\sharp} \otimes I$ can be also regarded as elements of $\Gamma_{A'/A}^{\sharp}$. Consider now the maps $r_{\bar{g}}=u(\bar{g})^{-1}\rho^{\ast}_{A'}(\bar{g})^{\sharp} \in \Gamma_{A'}^{\sharp}$. I claim that the map $G/N \to \Gamma_{A'}^{\#}$ given by $\bar{g} \mapsto r_{\bar{g}}$ is a group homomorphism. Indeed, by Lemma \ref{potwor}, we have \begin{eqnarray*}r_{\bar{g}}r_{\bar{h}}&=& u(\bar{g})^{-1}\rho^{\ast}_{A'}(\bar{g})^{\sharp}u(\bar{h})^{-1}\rho^{\ast}_{A'}(\bar{h})^{\sharp}=u(\bar{g})^{-1}\left(u(\bar{h})^{\bar{g}}\right)^{-1}\rho^{\ast}_{A'}(\bar{g})^{\sharp}\rho^{\ast}_{A'}(\bar{h})^{\sharp}\\ &=& u(\bar{g})^{-1}\left(u(\bar{h})^{\bar{g}}\right)^{-1}\eta(\bar{g},\bar{h})^{\sharp}\rho^{\ast}_{A'}(\overline{gh})^{\sharp} =  u(\overline{gh})^{-1}\rho^{\ast}_{A'}(\overline{gh})^{\sharp}=r_{\overline{gh}}. \end{eqnarray*} This shows that both $[r_{\bar{g}}]$ and $\kappa_{A'}^{G/N}$ are in the fiber of $D_{G/N}(A') \to D_{G/N}(A)$ over $\kappa_A^{G/N}$. Define $$\beta = \kappa^{G/N}_{A'}-[r_{\bar{g}}] \in \HO^1(G/N,\der^{\sharp}) \otimes I.$$
We are now ready to define the obstruction map. Put $$\underline{Obs}=\nu_e^{\Phi}\left(\kappa_A,(\kappa_{A'}^N,\kappa_{A'}^{G/N})\right)=\lambda+\gamma(\alpha)-\sigma(\beta) \in \HO^1(G/N,\der^{\sharp}\!/\der^N)\otimes I.$$

We need to verify that the definition of $\underline{Obs}$ does not depend on the choices made, i.e., on the choice of: \begin{enumerate} \item[(i)] the representative $\rho_{A}$ of the class $\kappa_A$, \item[(ii)] the class $\lambda \in \HO^1(G/N,\der^{\sharp}\!/\der^N) \otimes I$ such that $\nu_e(\kappa_A)=\inf(\partial(\lambda))$, \item[(iii)] the map $u\colon G/N \to \der^{\sharp} \otimes I$ such that $u(\id)=0$ and $\lambda=[\bar{u}_{\bar{g}}]$, \item[(iv)] the lift $\rho^{\ast}_{A'}$ of $\rho_A$ to $A'$ such that $\rho^{\ast}_{A'}(g)\rho^{\ast}_{A'}(h)\rho^{\ast}_{A'}(gh)^{-1}=\partial u(\bar{g},\bar{h})$. \end{enumerate} 

\emph{The class $\underline{Obs}$ is independent of the choices made in \textup{(ii)}--\textup{(iv)}\textup{:}} Let $\tilde{\lambda}, \tilde{u}, \tilde{\rho}_{A'}^{\ast}$ be another choices for (ii)--(iv), and denote the respective invariants by $\tilde{\alpha}$ and $\tilde{\beta}$. 
Since $\rho^{\ast}_{A'}$ and $\tilde{\rho}^{\ast}_{A'}$ are both lifts of $\rho_A$, we can write $\tilde{\rho}^{\ast}_{A'}(g) =\gamma_g \rho^{\ast}_{A'}(g)$ with $\gamma_g\in\Gamma_{A',A}$.

Directly from the definition we see that \begin{equation}\label{swiadomosc}\tilde{\lambda}-\lambda=[\overline{\tilde{u}}_{\bar{g}}]-[\overline{u}_{\bar{g}}].\end{equation}
Similarly, we have \begin{equation}\label{marnotrawstwo}\tilde{\alpha}-\alpha=[\rho^{\ast}_{A'}]-[\tilde{\rho}^{\ast}_{A'}]=-[\gamma_s].\end{equation}

We now proceed to compute $\tilde{\beta}-\beta$. The subrings of $A' \psl t \psr$ of $N$-invariants with respect to the two $N$-actions are, respectively, $A'\psl y_{A'} \psr$ and $A'\psl \tilde{y}_{A'}\psr$, where $$y_{A'}=\prod_{s\in N} \rho^{\ast}_{A'}(s)(t), \quad \tilde{y}_{A'}=\prod_{s\in N} \tilde{\rho}^{\ast}_{A'}(s)(t).$$ The elements $y_{A'}$ and $\tilde{y}_{A'}$ have the same image in $A\psl t\psr$, and the map $\tau$ mapping $y_{A'}$ to $\tilde{y}_{A'}$ gives an isomorphism of $A'\psl y_{A'} \psr$ and $A'\psl \tilde{y}_{A'}\psr$ over $A\psl y_A \psr$.
Consider the following diagram:  \begin{displaymath}\xymatrix{ A'\psl y_{A'}\psr \ar[rr]^{\tau} \ar[d]^{\rho_{A'}^{\ast}(\bar{g})^{\sharp}} && A'\psl \tilde{y}_{A'}\psr \ar[d]^{\tilde{\rho}_{A'}^{\ast}(\bar{g})^{\sharp}} \\  A'\psl y_{A'} \psr\ar[r]^{\upsilon_{\bar{g}}} \ar[d]^{u_{\bar{g}}^{-1}} &A'\psl y_{A'} \psr \ar[r]^{\tau} \ar[d]^{\tilde{u}_{\bar{g}}^{-1}} & A'\psl \tilde{y}_{A'}\psr \ar[d]^{\tilde{u}_{\bar{g}}^{-1}}  \\ A'\psl y_{A'} \psr\ar[r]^{\delta_{\bar{g}}}  &A'\psl y_{A'} \psr\ar[r]^{\tau}  & A'\psl \tilde{y}_{A'}\psr }\end{displaymath} in which all the maps are isomorphisms, and $\delta_{\bar{g}}$ and $\upsilon_{\bar{g}}$ are chosen in such a way the the diagram commutes. 
Tensoring the diagram with $A$, we see that $\delta_{\bar{g}} \otimes_{A'}\! A = \upsilon_{\bar{g}} \otimes_{A'}\! A = \id$. 
Thus we can think of $\delta_{\bar{g}}$ and $\upsilon_{\bar{g}}$ as elements of $\Gamma_{A'/A}\cong\der^{\sharp}\otimes I$. 
It is also clear that $$\tilde{\beta}-\beta=[u^{-1}_{\bar{g}}\rho^{\ast}_{A'}(\bar{g})^{\sharp}]- [\tilde{u}^{-1}_{\bar{g}}\tilde{\rho}^{\ast}_{A'}(\bar{g})^{\sharp}]=-[\delta_{\bar{g}}].$$ This gives \begin{equation}\label{wnikliwy}\sigma(\tilde{\beta})-\sigma(\beta)=-[\overline{\delta}_{\bar{g}}]=[\overline{\tilde{u}}_{\bar{g}}]-[\overline{u}_{\bar{g}}]-[\overline{\upsilon}_{\bar{g}}].\end{equation} We need to relate $\delta$ to $\gamma$ and $u$.
Recall $\tilde{\rho}^{\ast}_{A'}(g) =\gamma_g \rho^{\ast}_{A'}(g)$ and write $\gamma_g(t)=t+d_g(t)$ with $d_g\in \der \otimes I$. Recall also that we write $f_g=\rho(g)(t)$. Note first that we have \begin{eqnarray*} \tilde{y}_{A'}&=&\prod_{s\in N}\tilde{\rho}_{A'}^{\ast}(s)(t)=\prod_{s\in N}\gamma_{s}\rho_{A'}^{\ast}(s)(t)=\prod_{s\in N}\rho_{A'}^{\ast}(s)\left(t+d_{s}(f_{s})\right)\\&=&\prod_{s\in N}\rho_{A'}^{\ast}(s)(t) \left(1+\sum_{s\in N}\frac{d_s(f_s)}{f_s}\right)=y_{A'}+y\sum_{s\in N}\frac{d_s(f_s)}{f_s}. \end{eqnarray*} (In the last term, we write $y$ instead of $y_A$ or $\tilde{y}_A$ because the sum lies in $I$, and hence is annihilated by $\mathfrak{m}_{A'}$.) 
We now look at the top rectangle in the previous diagram. 
We have \begin{eqnarray}\nonumber \tilde{\rho}^{\ast}(\bar{g})^{\sharp} \tau(y_{A'}) &=&\tilde{\rho}^{\ast}(\bar{g})^{\sharp}(\tilde{y}_{A'}) = \prod_{s\in N}\tilde{\rho}_{A'}^{\ast}(gs)(t)=\prod_{s\in N}\gamma_{gs}\rho_{A'}^{\ast}(gs)(t) \\ \nonumber &=& \prod_{s\in N}\rho_{A'}^{\ast}(gs)\left(t+d_{gs}(f_{gs})\right) \\ \label{pentatlon} &=& \prod_{s\in N}\rho_{A'}^{\ast}(gs)(t)+y^g\sum_{s\in N}\frac{d_{gs}(f_{gs})}{f_{gs}}. \end{eqnarray} 
Similarly, write $\upsilon_{\bar{g}}(y_{A'})=y_{A'}+w_{\bar{g}}(y)$, $w_{\bar{g}} \in \der^{\sharp} \otimes I$. Then we have $$ \upsilon_{\bar{g}}\rho^{\ast}(\bar{g})^{\sharp}(y_{A'}) = \upsilon_{\bar{g}}\left(\prod_{s\in N}\rho_{A'}^{\ast}(gs)(t)\right)=\prod_{s\in N}\rho_{A'}^{\ast}(gs)(t) + w_{\bar{g}}\left(\prod_{s\in N}f_{gs}\right).$$ Now \begin{eqnarray}\nonumber \tau\upsilon_{\bar{g}}\rho^{\ast}(\bar{g})^{\sharp}(y_{A'}) &=&\tau\left(\prod_{s\in N}\rho_{A'}^{\ast}(gs)(t) + w_{\bar{g}}\left(\prod_{s\in N}f_{gs}\right)\right) \\ \nonumber &=& \prod_{s\in N}\rho_{A'}^{\ast}(gs)(t) + w_{\bar{g}}\left(\prod_{s\in N}f_{gs}\right)+\frac{dy^g}{dy}(\tilde{y}_{A'}-y_{A'})\\ \label{nieelastyczny} &=& \prod_{s\in N}\rho_{A'}^{\ast}(gs)(t) + \frac{dy^g}{dy}w_{\bar{g}}\left(y\right)+\frac{dy^g}{dy}y\sum_{s\in N}\frac{d_s(f_s)}{f_s}.\end{eqnarray} Since $\upsilon$ was chosen in such a way that the diagram commutes, comparing equations (\ref{pentatlon}) and (\ref{nieelastyczny}) we get \begin{equation} \label{sojusznik}w_{\bar{g}}\left(y\right)= y^g\left(\frac{dy^g}{dy}\right)^{-1}\sum_{s\in N}\frac{d_{gs}(f_{gs})}{f_{gs}} - y\sum_{s\in N}\frac{d_s(f_s)}{f_s}.\end{equation} 
Recall that we have \begin{align*} \rho^{\ast}_{A'}(g)\rho^{\ast}_{A'}(h)\rho^{\ast}_{A'}(gh)^{-1}&=\partial u(\bar{g},\bar{h}), \\ \tilde{\rho}^{\ast}_{A'}(g)\tilde{\rho}^{\ast}_{A'}(h)\tilde{\rho}^{\ast}_{A'}(gh)^{-1}&=\partial \tilde{u}(\bar{g},\bar{h}).\end{align*} Since $\tilde{\rho}_{A'}^{\ast}(g)=\gamma_g\rho^{\ast}_{A'}(g)$, Lemma \ref{potwor} gives $$\gamma_g\gamma_h^g\gamma_{gh}^{-1}=\partial \bar{u}(\bar{g},\bar{h}) - \partial u(\bar{g},\bar{h}).$$ In particular $$ \gamma_{gs}=\gamma_g \gamma_s^g \quad \text{for}\quad g\in G, s\in N,$$  and hence $$ d_{gs}=d_s^g+d_g \quad \text{for}\quad g\in G, s\in N.$$
This allows to use Proposition \ref{prefektura}. Together with equation (\ref{sojusznik}) it shows that there exists an $\alpha \in \der_1$ such that $$\alpha^g-\alpha=d_g(y_G)-\frac{dy_G}{dy}w_{\bar{g}}(y)=d_g(y_G)-w_{\bar{g}}(y_G),$$ where $y_G = \prod_{g\in G}f_g(t)$. In particular $$\alpha^s-\alpha=d_s(y_G) \quad \text{for}\quad s \in N.$$ Corollary \ref{architektura} gives then $\gamma([d_s(y_G)])=[\overline{w}_{\bar{g}}(y_G)]$, or in other notation $$\gamma([\gamma_s])=[\overline{\upsilon}_{\bar{g}}].$$ Combining (\ref{swiadomosc}), (\ref{marnotrawstwo}) and (\ref{wnikliwy}), this gives $$\tilde{\lambda}+\gamma(\tilde{\alpha})-\sigma(\tilde{\beta})= \lambda+\gamma(\alpha)-\sigma(\beta).$$ and hence the definition of $\underline{Obs}$ is independent of the choices for (ii)-(iv).

\emph{The class $\underline{Obs}$ is independent of the choice for \textup{(i)}\textup{:}} To prove that $\underline{Obs}$ is independent on the choice made in (i), choose a different representative $\tilde{\rho}_A$ of $\kappa_A$ and write $\tilde{\rho}(g)=\chi\rho_A(g)\chi^{-1}$ with $\chi \in \Gamma_{A,k}$. Choose $\chi'\in\Gamma_{A',k}$ lying over $\chi$. Now make the following choices: $$ \text{in (ii) put } \tilde{\lambda}=\lambda, \quad \text{in (iii) put } \tilde{u}=u, \quad \text{in (iv) put } \tilde{\rho}_{A'}^{\ast}(g)=\chi'\rho_{A'}^{\ast}(g)\chi'^{-1}.$$  It is then clear that $$\tilde{\lambda}=\lambda,\quad \tilde{\alpha}=\alpha,\quad \tilde{\beta}=\beta.$$ This finishes the proof that $\underline{Obs}$ is well-defined. 

\emph{Obstruction space\textup{:}} The fact that $\underline{Obs}$ depends functorially on the extension $$e\colon 0\to I \to A' \to A\to 0$$ is obvious from the construction. This proves that $\HO^2(G/N, \der^{\sharp}\!/\der^N)$ is an obstruction space to $\Phi$.

\emph{Completeness of the obstruction space\textup{:}} To prove that $\HO^2(G/N, \der^{\sharp}\!/\der^N)$ is complete, assume that $\underline{Obs}=0$. Then $$\lambda \in \im \gamma + \im \sigma$$ and hence by Theorem \ref{Mandelsztam}, $$\underline{obs}=\inf(\partial(\lambda))=0.$$ Hence there is no obstruction to lifting $\rho_A$ to $A'$ and there exists a group homomorphism $\rho_{A'} \colon G \to \Gamma_{A'}$, which is a lift of $\rho_A$ to $A'$. Since $\underline{Obs}$ does not depend on the choices made in (i)--(iv), we can make the following choices: $$ \text{in (ii) put } \lambda=0, \quad \text{in (iii) put } u=0, \quad \text{in (iv) put } \rho_{A'}^{\ast}=\rho_{A'}. $$ Consider $$\kappa_{A'}^0=[\rho_{A'}]\in D_G(A').$$ Then $\kappa_{A'}^0$ is a lift of $\kappa_A$ to $A'$ and $$\alpha=\kappa_{A'}^N-\res(\kappa_{A'}^0), \quad \beta=\kappa_{A'}^{G/N}-\ind(\kappa_{A'}^0).$$ Since $\gamma(\alpha)=\sigma(\beta)$, by Corollary \ref{Jugendstill}.ii there exists a $\zeta\in\HO^1(G,\der)$ such that $$\alpha=\res(\zeta), \quad \beta=\ind(\zeta).$$ Put $\kappa_{A'}=\kappa_{A'}^0+\zeta \in D_G(A')$. Then$$\res(\kappa_{A'})=\res(\kappa_{A'}^0)+\res(\zeta)=\kappa_{A'}^N, \quad \ind(\kappa_{A'})=\ind(\kappa_{A'}^0)+\ind(\zeta)=\kappa_{A'}^{G/N}.$$ Hence the image of $\kappa_{A'}$ by the map $$D_G(A')\to D_G(A)\times\left(D_N^{G/N}(A')\times D_{G/N}(A')\right)$$ is $\left(\kappa_A,\left(\kappa_{A'}^N,\kappa_{A'}^{G/N}\right)\right)$. This proves that the obstruction space is complete.\end{proof}

\begin{remark} Note that we can recover Theorem \ref{pyszczek} as a special case of Theorem \ref{szczezl}. In fact, when the order of $G/N$ is prime to $p$, the tangent and obstruction space to $D_{G/N}$ vanishes, and hence $D_{G/N}=\{\ast\}$ is a singleton functor. Similarly, the obstruction space to the map $\res \colon D_G \to D_N^{G/N}$ vanishes, and hence $\res$ is smooth. From this fact, we conclude as in the proof of Theorem \ref{pyszczek}. \end{remark}

\begin{lemma}\label{nieustepliwosc} Let $e \colon 0\to I \to A' \to A \to 0$ be a small extension in $\kate$ and let $\kappa \in D_G(A)$. Assume $D_N$ is pro-representable. Let $$\nu_e(\kappa) \in \HO^2(G,\der) \otimes I$$ be an obstruction to lifting $\kappa$ to $A'$. Then the image of $\kappa$ in $D_N^{G/N}(A)$ lifts to $D_N^{G/N}(A')$ if and only if the image of $\nu_e(\kappa)$ vanishes in $\HO^2(N,\der) \otimes I$ and in $\HO^1(G/N,\HO^1(N,\der))\otimes I$.
\end{lemma}
\begin{proof}
Choose a representative $\rho_A\colon G \to \Gamma_A$ of $\kappa\in D_G(A)$. The image of $\nu_e(\kappa)$ in $\HO^2(N,\der)\otimes I$ vanishes if and only if $\res(\kappa)\in D_N(A)$ admits a lift to $D_N(A')$. We can thus assume that this condition is satisfied and we can choose a lift $\rho_{A'}\colon N \to \Gamma_{A'}$ and its set-theoretic extension to a map $\rho^{\ast}_{A'}\colon G \to \Gamma_{A'}$ which lifts $\rho_A$. The deformation $\kappa'=[\rho_{A'}]\in D_N(A')$ is a lift of $\res(\kappa)$ to $A'$. The obstruction $\nu_e(\kappa)$ is the class of the 2-cocycle $\eta$ given by $$\rho_{A'}^{\ast}(g)\rho_{A'}^{\ast}(h)=\eta(g,h)\rho_{A'}^{\ast}(gh).$$

 Note that $\eta(s,t)=0$ for $s,t\in N$. We can then compute $g_{\ast}\kappa'$ as the class of $$g_{\ast}\rho_{A'}=\rho_{A'}^{\ast}(g)\rho_{A'}^{\ast}(g^{-1}sg)\rho_{A'}^{\ast}(g)^{-1}=\eta(g,g^{-1}sg)\eta(s,g)^{-1}\rho_{A'}^{\ast}(s).$$ This shows that \begin{equation}\label{grzyby} g_{\ast}\kappa'=\kappa'+[\eta(g,g^{-1}sg)-\eta(s,g)].\end{equation} Since $D_N$ is pro-representable, the fibers of $D_N(A') \to D_N(A)$ over $\kappa$ are torsors under the action of $T_{D_N}\cong\HO^1(N,\der)\otimes I$. Then there is a lift of $\res(\kappa)$ to $D_N^{G/N}(A')$ if and only if there exists a $\zeta \in T_{D_N}$ such that $$g_{\ast}(\kappa'-\zeta)=\kappa'-\zeta.$$ Together with (\ref{grzyby}) this gives $$\eta(g,g^{-1}sg)-\eta(s,g)=\zeta^g-\zeta.$$ The claim follows from Theorem \ref{nemezis}, which is phrased in a more general context in the next subsection.\end{proof}

\end{subsection}
\end{section}

\begin{section}{Appendix: A map in the Hochschild-Serre spectral sequence}

The purpose of this appendix is to provide an explicit formula for a certain map in the Hochschild-Serre spectral sequence in group cohomology. 
For lack of a reference, we present this computation here.

Recall that the Hochschild-Serre spectral sequence arises as the Grothendieck spectral sequence (cf. \cite[Theorem 5.8.3]{Weibel}) associated to the composition of functors of $N$- and $G/N$-invariants $$\Mod_G \xrightarrow{\mathrm{inv}_N} \Mod_{G/N} \xrightarrow{\mathrm{inv}_{G/N}} \mathrm{Ab}.$$ The spectral sequence takes on the form $$\HO^p(G/N,\HO^q(N,M)) \Rightarrow \HO^{p+q}(G,M).$$ For a general reference for spectral sequences, we use \cite[Appendix 3]{Eisenbud}.

\begin{thm}\label{nemezis} The map $$\ker \big( \HO^2(G,M) \to \HO^2(N,M)\big)\to \HO^1(G/N,\HO^1(N,M)) $$ induced from the Hochschild-Serre spectral sequence is given by $$[\eta(g,h)] \mapsto \xi_{\bar{g}}(s) = \eta(s, g) - \eta(g,g^{-1}sg) + gf(g^{-1}sg)- f(s),$$ where $f \colon N \to M$ is a cochain such that $$\eta(s,t)=(\partial f)(s,t) = sf(t)-f(st)+f(s), \quad s,t \in N.$$\end{thm}

\begin{proof}

Recall that group cohomology can be computed as the cohomology of the complex $\Hom_G(E_{\bullet},M)$, where $E_i$ is the free abelian group with the basis given by $(i+1)$-tuples $(g_0,\ldots,g_i)$ of elements of $G$. With the $G$-action given by $$g(g_0,\ldots,g_i)=(gg_0,\ldots,gg_i)$$ the module $E_i$ becomes a free $\mathbf{Z}[G]$-module. The derivation $d_i \colon E_i \to E_{i-1}$ is given by $$d_i (g_0,\ldots,g_i) = \sum_{j=0}^i (-1)^j (g_0,\ldots,\check{g}_j,\ldots,g_i).$$ 

We will change the notation slightly, replacing the complex $E_{\bullet}$ by an isomorphic complex $F_{\bullet}$ with less apparent symmetry. Denote by $F_i$ a free $\mathbf{Z}[G]$-module with basis given by $i$-tuples $[g_1,\ldots,g_i]$ and define the derivation $d_i \colon F_i \to F_{i-1}$ by $$d_i [g_1,\ldots,g_i] = g_1 [g_2,\ldots,g_i] + \sum_{j=1}^{i-1} (-1)^j [g_0,\ldots,g_jg_{j+1},\ldots,g_i] + (-1)^i [g_1,\ldots,g_{i-1}].$$ This is an isomorphic complex and the isomorphism is defined on the basis by $$[g_1,\ldots,g_i] \mapsto (1,g_1,g_1g_2,\ldots,g_1g_2\cdots g_i).$$ We shall call $E_{\bullet}$ and $F_{\bullet}$ the homogenous and the inhomogenous standard complexes, respectively. When it seems important to denote the group explicitly, we write $E_{\bullet}(G)$ and $F_{\bullet}(G)$.

Define $$E_0^{p,q} = \Hom_G(E_p(G/N),\Hom_N(E_q(G),M)).$$ This defines a double complex, with total complex $T$. Considering its vertical filtration, we get a spectral sequence $$_{\mathrm{vert}}E_2^{p,q}=\HO^p(G/N,\HO^q(N,M)) \Rightarrow \HO^{p+q}(T).$$ Considering its horizontal filtration, we also get a spectral sequence $$_{\mathrm{hor}}E_2^{p,q}=\HO^p(G/N,\HO^q(N,M)) \Rightarrow \HO^{p+q}(T).$$ Furthermore, the horizontal filtration spectral sequence satisfies $$_{\mathrm{hor}}E_2^{p,q}=\HO^q(G,M) \text{ for } p=0, \quad _{\mathrm{hor}}E_2^{p,q}=0 \text{ for } p \geq 1.$$ This shows that the spectral sequence degenerates on the level of $_{\mathrm{hor}}E_2^{p,q}$ and we get isomorphisms $\HO^q(G,M) \cong \HO^q(T)$. In this way we get a spectral sequence $$\mathrm{H}^p(G/N,\mathrm{H}^q(N,M)) \Rightarrow \mathrm{H}^{p+q}(G,M),$$ which is precisely the Hochschild-Serre spectral sequence. 

An element $a \in \HO^2(G,M)$ can be written as an element of $\Hom_G(E_2,M)$, namely as a 2-cocycle $\eta(g,h,k)$ with values in $M$. We easily see that the same element of $\HO^2(G,M)$ is induced by an element of $$E_0^{0,2}=\Hom_G(E_0(G/N),\Hom_N(E_2(G),M))$$ given by the map $$\bar{\gamma} \mapsto \big((g,h,k) \mapsto \eta(g,h,k)\big).$$ If $a$ lies in $\ker(\HO^2(G,M) \to \HO^2(N,M))$, then there exists a 1-cocycle $f(g,h) \in \Hom_N(E_1(N),M)$ such that $$\eta(g,h,k)=\partial f(g,h,k) = f(h,k)-f(g,k)+f(g,h).$$ Choose a section $s\colon G/N \to G$ and put $\sigma(g)=gs(\bar{g})^{-1}$. We can assume that $s(\id)=\id$ and hence $\sigma(\id)=\id$. Then $\sigma \colon G \to N$ is an $N$-equivariant map. Let $b$ be the element of $E_0^{0,1}$ given by $$\bar{\gamma} \mapsto \big((g,h) \mapsto (\gamma \xi)(g,h) = \gamma \xi(\gamma^{-1}g,\gamma^{-1}h)\big)$$ where $$\xi(g,h) = \eta(\sigma(g),g,h) + \eta(\sigma(h),\sigma(g),h) - f(\sigma(h),\sigma(g)).$$
One easily checks directly that $d_{\mathrm{vert}}b = a$. Further, $c = d_{\mathrm{hor}}b \in$ $E_0^{1,1}$ is given by $$(\bar{\gamma},\bar{\delta}) \mapsto \big( (g,h) \mapsto \zeta_{(\gamma,\delta)}(g,h) = (\delta\xi)(g,h)-(\gamma\xi)(g,h)\big).$$

We are now ready to compute the desired map. The element $a$ induces a class in $\HO^2(G,M)=E_2^{0,2}$. It also induces a class in the cohomology of the total complex $\HO^2(T)$.
This class is the same as the class of $$a-d_{\mathrm{tot}}b=a-d_{\mathrm{vert}}b+d_{\mathrm{hor}}b=c \in E_0^{1,1}.$$ Further, the class of $c$ in $$E_2^{1,1} = \HO^1(G/N,\HO^1(N,M))$$ is given as the class of the map in $$\Hom_G(E_1(G/N),\Hom_N(E_1(N),M))$$ given by $$(\bar{\gamma},\bar{\delta}) \mapsto \big( (g,h) \mapsto \zeta_{(\gamma,\delta)}(g,h) = (\gamma\xi)(g,h)-(\delta\xi)(g,h)\big).$$ (Note the restriction from $G$ to $N$.) 
Changing the notation from the complex $E_{\bullet}$ to the (isomorphic) inhomogenous complex $F_{\bullet}$ we get the map $$\bar{\delta} \mapsto \big( h \mapsto \zeta_{\delta}(h) \big),$$ with \begin{eqnarray*} \zeta_{\delta}(h) &=& \delta \eta(\sigma(\delta^{-1}),\delta^{-1},\delta^{-1}h) + \delta\eta(\sigma(\delta^{-1}h),\sigma(\delta^{-1}),\delta^{-1}h) - \delta f(\sigma(\delta^{-1}h),\sigma(\delta^{-1}))  \\ & & - \eta(\sigma(\id),\id,h) - \eta(\sigma(h),\sigma(\id),h) + f(\sigma(h),\sigma(\id)).\end{eqnarray*}
We preserve the notation $\eta(g,h)$ and $f(h)$ also for the associated inhomogenous cochains. 
In this notation, we get \begin{align*} \zeta_{\delta}(h) =& \delta\sigma(\delta^{-1}) \eta(\sigma(\delta^{-1})^{-1}\delta^{-1},h) + \delta\sigma(\delta^{-1}h)\eta(\sigma(\delta^{-1}h)^{-1}\sigma(\delta^{-1}),\sigma(\delta^{-1})^{-1}\delta^{-1}h) \\ &- \delta \sigma(\delta^{-1}h) f((\sigma(\delta^{-1}h))^{-1}\sigma(\delta^{-1})) -  \eta(\id,h) - h\eta(h^{-1},h) + hf(h^{-1}).\end{align*}
The map does not depend on the choice of a representative of $\delta$. Thus we can assume that $\sigma(\delta^{-1})=\id$. It follows that $\sigma(\delta^{-1}h)=\delta^{-1}h\delta$ and the map takes the form \begin{eqnarray}\label{redaktor} \zeta_{\delta}(h) &=& \delta \eta(\delta^{-1},h) +h \delta\eta(\delta^{-1}h^{-1}\delta,\delta^{-1}h) - h \delta f(\delta^{-1}h^{-1}\delta) \\ \nonumber & &-  \eta(\id,h) - h\eta(h^{-1},h) + h f(h^{-1}).\end{eqnarray} Applying the 2-cocycle equation $$g\eta(h,k)-\eta(gh,k)+\eta(g,hk)-\eta(g,h)=0, \quad g,h,k\in G$$ to the triples $(g,h,k)$ equal to $$(\delta,\delta^{-1},h), (h\delta,\delta^{-1}h^{-1}\delta,\delta^{-1}h), (\delta,\delta^{-1}h\delta,\delta^{-1}h^{-1}\delta), (h,\delta,\delta^{-1}), (\delta,\id,\id),$$ we get the following equations: \begin{align} \label{kolchoz1} &\delta\eta(\delta^{-1},h)-\eta(\id,h)+\eta(\delta,\delta^{-1}h)-\eta(\delta,\delta^{-1}) =0 \\\label{kolchoz2} &h\delta\eta(\delta^{-1}h^{-1}\delta,\delta^{-1}h)-\eta(\delta,\delta^{-1}h)+\eta(h\delta,\delta^{-1})-\eta(h\delta,\delta^{-1}h^{-1}\delta)=0 \\ \label{kolchoz3}&\delta\eta(\delta^{-1}h\delta,\delta^{-1}h^{-1}\delta)-\eta(h\delta,\delta^{-1}h^{-1}\delta)+\eta(\delta,\id)-\eta(\delta,\delta^{-1}h\delta)=0 \\\label{kolchoz4}&h\eta(\delta,\delta^{-1})-\eta(h\delta,\delta^{-1})+\eta(h,\id)-\eta(h,\delta)=0 \\ \label{kolchoz5}&\delta\eta(\id,\id)-\eta(\delta,\id)=0{}.\end{align} Adding equations (\ref{kolchoz1}), (\ref{kolchoz2}) and (\ref{kolchoz4}) and subtracting equations (\ref{kolchoz3}) and (\ref{kolchoz5}) we get \begin{align*}\delta\eta(\delta^{-1},h)+h\delta\eta(\delta^{-1}h\delta,&\delta^{-1}h)=\delta\eta(\delta^{-1}h\delta,\delta^{-1}h^{-1}\delta)-\eta(\delta,\delta^{-1}h\delta)+\delta\eta(\id,\id) \\ &+\eta(\id,h)+\eta(\delta,\delta^{-1})-h\eta(\delta,\delta^{-1})-\eta(h,\id)+\eta(h,\delta).\end{align*} Together with equation (\ref{redaktor}), this gives $$ \zeta_{\delta}(h) = \eta(h, \delta) - \eta(\delta,\delta^{-1}h\delta) +\eta(\delta,\delta^{-1}) - h \eta(\delta,\delta^{-1}) +\delta f(\delta^{-1}h\delta)- f(h).$$ Passing to cohomology, one sees that this gives precisely the map stated in the claim.
\end{proof}
%Using the fact that $$\eta(s,t)=sf(t)-f(st)+f(s) \quad \text{for}\quad s,t \in N,$$ equation (\ref{redaktor}) gives $$ \zeta_{\delta}(h) = \eta(h, \delta) - \eta(\delta,\delta^{-1}h\delta) +\eta(\delta,\delta^{-1}) - h \eta(\delta,\delta^{-1}) +\delta f(\delta^{-1}h\delta)- f(h).$$

\end{section}


\begin{thebibliography}{00}



%\bibitem{Barr} Michael Barr and Charles Wells, \emph{Toposes, Triples and Theories},  Reprints in Theory and Applications of Categories  No.\ 12  (2005), x+288 pp. [Corrected reprint of Grundl.\ Math.\ Wiss.\ vol.\ 278, Springer-Verlag, 1985].

\bibitem{Maugeais} Jos\'e Bertin, S. Maugeais, \emph{D\'{e}formations \'equivariantes des courbes semistables}, Annales de L'Institut Fourier \textbf{55} (2005), no.~6, pp.~1905--1941.
\bibitem{BM} Jos\'e Bertin, Ariane M\'{e}zard, \emph{D\'{e}formations formelles
des rev\^{e}tements sauvagement ramifi\'{e}s de courbes
alg\'{e}briques}, Invent. Math. \textbf{141} (2000), no.~1, pp.~195--238.

\bibitem{Mezard2} Jos\'e Bertin, Ariane M\'{e}zard, \emph{ Problem of formation of quotients and base change}, Manuscripta Math. \textbf{115} (2004), no.~4, pp.~467--487.


%\bibitem{BleherChinburg} Frauke M. Bleher, Ted Chinburg, \emph{Universal deformation rings need not be complete intersections}, Comptes Rendus Mathematique Acad. Sci. Paris, \textbf{342} (2006), pp.~229--232.

%\bibitem{Bleher2} Frauke M. Bleher, Ted Chinburg, \emph{Universal deformation rings need not be complete intersections}, Math. Ann. \textbf{337} (2007), no.~4, pp.~739--767.

%\bibitem{Boeckle} G. B\"{o}ckle, \emph{Demu\v{s}kin groups with group actions and applications to deformations of Galois representations}, Composition Math. \textbf{121} (2000), no.~2, pp.~109--154.


%\bibitem{Bourbakialgebre5} Nicolas Bourbaki, \emph{Alg\`{e}bre}, Chap. 5, Act. Sc. Ind. 1102, Hermann, 1959.

%\bibitem{Bourbaki} Nicolas Bourbaki, \emph{Alg\`{e}bre commutative}, Chap. 7, Act. Sc. Ind. 1314, Hermann, 1965.

%\bibitem{Bo} Nicolas Bourbaki, \emph{\'El\'ements de math\'emathique. I.II.I. Structures alg\'ebriques.} Actual. Sci. Ind. \textbf{934}, Hermann, 1942.

%\bibitem{Bouw} Irene Bouw, Stefan Wewers, \emph{The local lifting problem for dihedral groups}, Duke Math. J. \textbf{134} (2006), no.~3 , pp.~421--452. 

%\bibitem{Byszewski} Jakub Byszewski, \emph{A universal deformation ring which is not a complete intersection ring}, Comptes Rendus Mathematique Acad. Sci. Paris  \textbf{343} (2006), no. 9, pp. 565--568. 

\bibitem{BC} Jakub Byszewski, Gunther Cornelissen, \emph{Which weakly ramified group actions admit a universal formal deformation?}, Annales de L'Institut Fourier \textbf{59} (2009) no.\ 3, pp.~877--902. 


\bibitem{BCK} J. Byszewski, Gunther Cornelissen, Fumiharu Kato, \emph{Un anneau de deformation universel en conducteur superieur}, submitted for publication, \verb+arxiv.org/abs/0910.2557+. 

%\bibitem{Chinburg} Ted Chinburg, \emph{Can deformation rings of group representations not be local complete intersections?}, Rendiconti del Seminario Matematico della Universit\`a di Padova \textbf{113} (2005), p.~135.

\bibitem{CGH} Ted Chinburg, Robert Guralnick, David Harbater, \emph{Oort groups and lifting problems}, Compos.\ Math.\  \textbf{144} (2008),  no.\ 4, pp.~849--866.


%\bibitem{Canonaco} A. Canonaco, \emph{Introduction to algebraic stacks}, preliminary version, $29^{\text{th}}$ April 2005, \verb+www.mat.uniroma1.it/seminari/geo-alg/dispense/stacks.pdf+.

\bibitem{Cornelissen} Gunther Cornelissen, Fumiharu Kato,
\emph{Equivariant deformation of Mumford curves and of ordinary
curves in  positive characteristic},  Duke Math. J. \textbf{116}
(2003), no.~3, pp.~431--470.


%\bibitem{CKgerman} Gunther Cornelissen, Fumiharu Kato, \emph{Zur {E}ntartung schwach verzweigter {G}ruppenoperationen auf   {K}urven}, J. Reine Angew. Math. \textbf{589} (2005), pp.~201--236.

%\bibitem{CM} Gunther Cornelissen, Ariane M{\'e}zard, \emph{Rel\`evements des   rev\^etements de courbes faiblement ramifi\'es}, Math. Z. \textbf{254}  (2006), no.~2, pp.~239--255.

%\bibitem{Dokchitser} Tim Dokchitser, \emph{Deformations of $p$-divisible groups and $p$-descent on elliptic curves}, Ph.\ D.\ thesis, \\ \verb+igitur-archive.library.uu.nl/dissertations/1922434/inhoud.htm+.


%\bibitem{Dokchitserarxiv} Tim Dokchitser, \emph{Quotients of functors of {A}rtin rings}, Mathematical Proceedings of the Cambridge Philosophical Society, published online by Cambridge University Press on 18 Dec 2008.


\bibitem{Eisenbud} David Eisenbud, \emph{Commutative algebra with a view toward algebraic geometry},  Graduate Texts in Mathematics \textbf{150}, Springer-Verlag, 2004.


\bibitem{Fantechi}
Barbara Fantechi, Marco Manetti, \emph{Obstruction calculus of functors of {A}rtin rings, I}, Journal of Algebra \textbf{202} (1998), no.~2, pp.~541--576.


%\bibitem{Green} Barry Green, Michel Matignon, \emph{Liftings of Galois Covers of Smooth Curves}, Compositio Math. \textbf{113} (1998), no.~3, pp.~239--274.  

%\bibitem{Green2} Barry Green, Michel Matignon, \emph{Order $p$ automorphisms of the open disc of a $p$-adic field}, J. Amer. Math. Soc.  \textbf{12}  (1999), no.~1, pp.~269--303. 

% \bibitem{Groth} Alexander Grothendieck, \emph{Technique de descente et th\'eor\`emes   d'existence en g\'eom\'etrie alg\'ebrique. II. Le th\'eor\`eme   d'existence en th\'eorie formelle des modules}, S\'eminaire Bourbaki \textbf{12} (1959/60), no.\ 195, pp.~369--390.


%\bibitem{Grothendiecklist}  Pierre Colmez, Jean-Pierre Serre, \emph{Correspondance Grothendieck-Serre}, Documents Math\'ematiques \textbf{2} (2001), Soci\'et\'e Math\'ematique de France.



%\bibitem{Tohoku} Alexander Grothendieck, \emph{Sur quelques points d'alg\`ebre homologique}. T\^{o}hoku Mathemathical Journal \textbf{9} (1957), pp.~119--221.

%\bibitem{Har} Robin Hartshorne, \emph{Algebraic geometry},  Graduate Texts in Mathematics \textbf{52}, Springer-Verlag, 1977.

% \bibitem{Har2} Robin Hartshorne, \emph{Lectures on Deformation Theory}, 2004, \\ \verb+math.berkeley.edu/+ \verb+~robin/math274root.pdf+.


\bibitem{Hochschild} Gerhard Hochschild, Jean-Pierre Serre, \emph{Cohomology of group extensions}, Trans. Amer. Math. Soc. \textbf{74} (1953), pp.~110--134.

%\bibitem{Liu} Qing Liu, \emph{Algebraic Geometry and Arithmetic Curves}, Oxford University Press, 2002.

%\bibitem{Matsumura} Hideyuki Matsumura, \emph{Commutative ring theory}, Cambridge Studies in Advanced Mathematics \textbf{8}, Cambridge University Press, 1986.

%\bibitem{Mazur} Barry Mazur, \emph{An introduction to the deformation theory of Galois representations}, in: G. Cornell, J.H.Silverman, G. Stevens (Eds.), \emph{Modular forms and Fermat's last Theorem}, Springer-Verlag 1997, pp.~243--312.


%\bibitem{Mazur 1989} Barry Mazur, \emph{Deforming Galois Representations}, pp.~385--438, in Y.Ihara, K. Ribet, J.- P. Serre, \emph{Galois Groups over $\mathbf{Q}$}, Springer-Verlag, 1989.


%\bibitem{MacLane} Saunders Mac Lane, \emph{Homology}. Springer-Verlag, 1963.



% \bibitem{Milne} James S. Milne, \emph{Class field theory}, \\ \verb+www.jmilne.org/math/CourseNotes/math776.html+.


 \bibitem{Nakajima} Sh{\=o}ichi Nakajima, \emph{{$p$}-ranks and automorphism groups of algebraic curves}, Trans. Amer. Math. Soc. \textbf{303} (1987), no.~2, pp.~595--607.
  
\bibitem{Neukirch} J\"urgen Neukirch, Alexander Schmidt, Kay Wingberg, \emph{Cohomology of number fields}, Grundlehren der Mathematischen Wissenschaften \textbf{323}, Springer-Verlag, 2000.

% \bibitem{Nottingham} Marcus du~Sautoy, Dan Segal, and Aner Shalev (eds.), \emph{New horizons in   pro-{$p$} groups}, Progress in Mathematics, vol.~184, Birkh\"auser Boston   Inc., Boston, 2000.


\bibitem{Schlessinger}
Michael Schlessinger, \emph{Functors of {A}rtin rings}, Trans.\ Amer.\ Math.\ Soc.\
  \textbf{130} (1968), pp.~208--222.
  
% \bibitem{SSO} Tsutomu Sekiguchi, Frans Oort, Noriyuki Suwa, \emph{On the deformation of Artin-Schreier to Kummer}, Annales scientifiques de L'\'Ecole Normale Sup\'erieure (4) \textbf{22} (1989), no.~3, pp.~345--375. 

%\bibitem{Sernesi} Edoardo Sernesi, \emph{Deformations of algebraic schemes}, Grundlehren der   Mathematischen Wissenschaften \textbf{334}, Springer-Verlag, 2006.

\bibitem{Serre} Jean-Pierre Serre, \emph{Corps locaux}, Hermann, 1968.

\bibitem{Weibel} Charles A. Weibel, \emph{An Introduction to Homological Algebra}, Cambridge University Press, 1995.


\end{thebibliography}
\end{document}